\pdfoutput=1
\documentclass[12pt, reqno]{amsart}

\usepackage[margin=1.5in]{geometry}

\usepackage[utf8]{inputenc}
\usepackage[T1]{fontenc}
\usepackage[english]{babel}

\usepackage{graphicx}
\usepackage{tikz-cd}

\usepackage{etoolbox}

\makeatletter
\let\old@tocline\@tocline
\let\section@tocline\@tocline
\newcommand{\subsection@dotsep}{4.5}
\newcommand{\subsubsection@dotsep}{4.5}
\patchcmd{\@tocline}
  {\hfil}
  {\nobreak
     \leaders\hbox{$\m@th
        \mkern \subsection@dotsep mu\hbox{.}\mkern \subsection@dotsep mu$}\hfill
     \nobreak}{}{}
\let\subsection@tocline\@tocline
\let\@tocline\old@tocline

\patchcmd{\@tocline}
  {\hfil}
  {\nobreak
     \leaders\hbox{$\m@th
        \mkern \subsubsection@dotsep mu\hbox{.}\mkern \subsubsection@dotsep mu$}\hfill
     \nobreak}{}{}
\let\subsubsection@tocline\@tocline
\let\@tocline\old@tocline

\let\old@l@subsection\l@subsection
\let\old@l@subsubsection\l@subsubsection

\def\@tocwriteb#1#2#3{%
  \begingroup
    \@xp\def\csname #2@tocline\endcsname##1##2##3##4##5##6{%
      \ifnum##1>\c@tocdepth
      \else \sbox\z@{##5\let\indentlabel\@tochangmeasure##6}\fi}%
    \csname l@#2\endcsname{#1{\csname#2name\endcsname}{\@secnumber}{}}%
  \endgroup
  \addcontentsline{toc}{#2}%
    {\protect#1{\csname#2name\endcsname}{\@secnumber}{#3}}}%

\newlength{\@tocsectionindent}
\newlength{\@tocsubsectionindent}
\newlength{\@tocsubsubsectionindent}
\newlength{\@tocsectionnumwidth}
\newlength{\@tocsubsectionnumwidth}
\newlength{\@tocsubsubsectionnumwidth}
\newcommand{\settocsectionnumwidth}[1]{\setlength{\@tocsectionnumwidth}{#1}}
\newcommand{\settocsubsectionnumwidth}[1]{\setlength{\@tocsubsectionnumwidth}{#1}}
\newcommand{\settocsubsubsectionnumwidth}[1]{\setlength{\@tocsubsubsectionnumwidth}{#1}}
\newcommand{\settocsectionindent}[1]{\setlength{\@tocsectionindent}{#1}}
\newcommand{\settocsubsectionindent}[1]{\setlength{\@tocsubsectionindent}{#1}}
\newcommand{\settocsubsubsectionindent}[1]{\setlength{\@tocsubsubsectionindent}{#1}}

\renewcommand{\l@section}{\section@tocline{1}{\@tocsectionvskip}{\@tocsectionindent}{}{\@tocsectionformat}}%
\renewcommand{\l@subsection}{\subsection@tocline{2}{\@tocsubsectionvskip}{\@tocsubsectionindent}{}{\@tocsubsectionformat}}%
\renewcommand{\l@subsubsection}{\subsubsection@tocline{3}{\@tocsubsubsectionvskip}{\@tocsubsubsectionindent}{}{\@tocsubsubsectionformat}}%
\newcommand{\@tocsectionformat}{}
\newcommand{\@tocsubsectionformat}{}
\newcommand{\@tocsubsubsectionformat}{}
\expandafter\def\csname toc@1format\endcsname{\@tocsectionformat}
\expandafter\def\csname toc@2format\endcsname{\@tocsubsectionformat}
\expandafter\def\csname toc@3format\endcsname{\@tocsubsubsectionformat}
\newcommand{\settocsectionformat}[1]{\renewcommand{\@tocsectionformat}{#1}}
\newcommand{\settocsubsectionformat}[1]{\renewcommand{\@tocsubsectionformat}{#1}}
\newcommand{\settocsubsubsectionformat}[1]{\renewcommand{\@tocsubsubsectionformat}{#1}}
\newlength{\@tocsectionvskip}
\newcommand{\settocsectionvskip}[1]{\setlength{\@tocsectionvskip}{#1}}
\newlength{\@tocsubsectionvskip}
\newcommand{\settocsubsectionvskip}[1]{\setlength{\@tocsubsectionvskip}{#1}}
\newlength{\@tocsubsubsectionvskip}
\newcommand{\settocsubsubsectionvskip}[1]{\setlength{\@tocsubsubsectionvskip}{#1}}

\patchcmd{\tocsection}{\indentlabel}{\makebox[\@tocsectionnumwidth][l]}{}{}
\patchcmd{\tocsubsection}{\indentlabel}{\makebox[\@tocsubsectionnumwidth][l]}{}{}
\patchcmd{\tocsubsubsection}{\indentlabel}{\makebox[\@tocsubsubsectionnumwidth][l]}{}{}

\newcommand{\@sectypepnumformat}{}
\renewcommand{\contentsline}[1]{%
  \expandafter\let\expandafter\@sectypepnumformat\csname @toc#1pnumformat\endcsname%
  \csname l@#1\endcsname}
\newcommand{\@tocsectionpnumformat}{}
\newcommand{\@tocsubsectionpnumformat}{}
\newcommand{\@tocsubsubsectionpnumformat}{}
\newcommand{\setsectionpnumformat}[1]{\renewcommand{\@tocsectionpnumformat}{#1}}
\newcommand{\setsubsectionpnumformat}[1]{\renewcommand{\@tocsubsectionpnumformat}{#1}}
\newcommand{\setsubsubsectionpnumformat}[1]{\renewcommand{\@tocsubsubsectionpnumformat}{#1}}
\renewcommand{\@tocpagenum}[1]{%
  \hfill {\mdseries\@sectypepnumformat #1}}

\let\oldappendix\appendix
\renewcommand{\appendix}{%
  \leavevmode\oldappendix%
  \addtocontents{toc}{%
    \protect\settowidth{\protect\@tocsectionnumwidth}{\protect\@tocsectionformat\sectionname\space}%
    \protect\addtolength{\protect\@tocsectionnumwidth}{2em}}%
}
\makeatother

\makeatletter
\settocsectionnumwidth{2em}
\settocsubsectionnumwidth{2.5em}
\settocsubsubsectionnumwidth{3em}
\settocsectionindent{1pc}%
\settocsubsectionindent{\dimexpr\@tocsectionindent+\@tocsectionnumwidth}%
\settocsubsubsectionindent{\dimexpr\@tocsubsectionindent+\@tocsubsectionnumwidth}%
\makeatother

\settocsectionvskip{5pt}
\settocsubsectionvskip{0pt}
\settocsubsubsectionvskip{0pt}
    
\settocsectionformat{\bfseries}
\settocsubsectionformat{\mdseries}
\settocsubsubsectionformat{\mdseries}
\setsectionpnumformat{\bfseries}
\setsubsectionpnumformat{\mdseries}
\setsubsubsectionpnumformat{\mdseries}

\let\oldtableofcontents\tableofcontents
\renewcommand{\tableofcontents}{%
  \vspace*{-\linespacing}% Default gap to top of CONTENTS is \linespacing.
  \oldtableofcontents}

\usepackage{amsmath, amsthm, amssymb, amsfonts}
\usepackage{enumitem}
\usepackage{mathrsfs}
\usepackage{amsaddr}
\usepackage{import}
\usepackage{wrapfig}
\usepackage{xcolor}
\definecolor{antiquefuchsia}{rgb}{0.57, 0.36, 0.51}
\definecolor{azure}{rgb}{0.0, 0.5, 1.0}
\definecolor{blue(ncs)}{rgb}{0.0, 0.53, 0.74}

\usepackage{tabularx}
\usepackage[colorlinks = true, linkcolor = blue(ncs), citecolor = antiquefuchsia]{hyperref}

\usepackage{comment}
\usepackage{todonotes}

\theoremstyle{plain}
\newtheorem{theorem}{Theorem}
\newtheorem{proposition}[theorem]{Proposition}
\newtheorem{corollary}[theorem]{Corollary}
\newtheorem{lemma}[theorem]{Lemma}

\theoremstyle{definition}
\newtheorem{definition}[theorem]{Definition}
\newtheorem{example}[theorem]{Example}

\theoremstyle{remark}
\newtheorem{remark}{Remark}

\numberwithin{theorem}{section}
\numberwithin{equation}{section}

\newcommand{\R}{\mathbb{R}}

\newcommand{\Z}{\mathbb{Z}}
\newcommand{\U}{\mathcal{U}}

\newcommand{\B}{\mathcal{B}}
\newcommand{\so}{\mathfrak{so}}
\newcommand{\I}{\mathcal{I}}

\newcommand{\A}{\mathcal{A}}
\renewcommand{\sc}{\mathrm{sc}}
\newcommand{\g}{\mathfrak{g}}
\newcommand{\gl}{\mathfrak{gl}}
\newcommand{\D}{\mathscr{D}}

\newcommand{\M}{\mathcal{M}}
\newcommand{\E}{\mathscr{E}}

\newcommand{\id}{\mathrm{id}}

\newcommand{\inv}{\mathrm{inv}}

\renewcommand{\L}{\mathcal{L}}

\newcommand{\C}{\mathbb{C}}
\renewcommand{\u}{\mathfrak{u}}

\newcommand{\defn}{\overset{\mathrm{def}}{=} }

\DeclareMathOperator{\supp}{supp}

\DeclareMathOperator{\End}{End}

\DeclareMathOperator{\Gl}{Gl}

\DeclareMathOperator{\F}{\mathcal{F}}
\DeclareMathOperator{\dist}{dist}

\renewcommand{\d}{\mathrm{d}}
\newcommand{\interior}{\mathrm{int}}

\renewcommand{\newline}{\vspace{.2em}\\\indent}

\begin{document}
\title{Stability of the non-abelian $X$-ray transform in dimension $\ge 3$}
\author{Jan Bohr}
\address{University of Cambridge\,${}^\dagger$\footnote{$\dagger$ {\it Department of Pure Mathematics and Mathematical Statistics}, Wilberforce Road,  Cambridge CB3 0WB, UK}}%\address{{\small Department of Pure Mathematics and Mathematical Statistics\\ University of Cambridge,  Cambridge CB3 0WB, UK}}
\date{\today}
\email{bohr@maths.cam.ac.uk}

%\renewcommand{\contentsnamefont}{\bfseries\Large}

%\vspace{-20mm}

%\vspace{-10mm}
\setcounter{tocdepth}{1}

\begin{abstract}
Non-abelian $X$-ray tomography seeks to recover a matrix potential $\Phi:M\rightarrow \C^{m\times m}$ in a domain $M$ from measurements of its so called scattering data $C_\Phi$ at $\partial M$. For $\dim M\ge 3$ (and under appropriate convexity and regularity conditions), injectivity of the forward map $\Phi \mapsto C_\Phi$ was established in \cite{PSUZ19}. In this article we extend \cite{PSUZ19} by proving a H{\"o}lder-type stability estimate. As an application we generalise a statistical consistency result for $\dim M =2$  \cite{MNP19} to higher dimensions.\\
The injectivity proof in \cite{PSUZ19} relies on a novel method by Uhlmann-Vasy \cite{UhVa16}, which first establishes injectivity in a shallow layer below $\partial M$ and then globalises this by a layer stripping argument. The main technical contribution of this paper is a more quantitative version of these arguments, in particular proving uniform bounds on layer-depth and stability constants.
\end{abstract}

\maketitle

\tableofcontents 

\newpage 
\section{Introduction}
Let $(M,g)$ be a compact $d$-dimensional Riemannian manifold with strictly convex boundary ($d\ge 2$) and $\Phi:M\rightarrow \C^{m\times m}$ ($m\ge 1$) a continuous matrix-potential. Suppose $\gamma: [0,\tau]\rightarrow M$ is a unit-speed geodesic with endpoints on $\partial M$ and consider the linear matrix  differential equation
\begin{equation}
\dot U(t) + \Phi(\gamma(t))U(t) = 0,\quad U(\tau) = \id.
\end{equation}
This has a unique continuous solution $U:[0,\tau]\rightarrow \Gl(m,\C)=\{A\in \C^{m\times m}:\det A \neq 0\}$ and we write $C_\Phi(\gamma)=U(0)\in \Gl(m,\C)$ for its value at the boundary. The matrix $C_\Phi(\gamma)$ is called {\it scattering data} or {\it non-abelian $X$-ray transform} of $\Phi$  (along $\gamma$).  For $m=1$ we have $\log C_\Phi(\gamma)=  \int_0^\tau \Phi(\gamma(t)) \d t$,  which is the standard $X$-ray transform; for $m\ge 2$ this relation breaks as $\Gl(m,\C)$ ceases to be abelian. 
\\
We are concerned with an inverse problem for the non-abelian $X$-ray transform with access to partial data: Can one recover $\Phi$  in an open set $O\subset M$ from measuring $C_\Phi(\gamma)$ for geodesics $\gamma$ that do not leave $O$? For $d\ge 3$ %, in which case the  problem is formally overdetermined, 
and $O\subset M$ satisfying  the so called \textit{foliation condition} (see Definition \ref{foliationcondition} below) it is known that locally, smooth potentials are determined uniquely by their scattering data. Precisely,  \cite{PSUZ19} establishes injectivity of the map
\begin{equation} \label{def1nonabelian}
C^\infty(O,\C^{m\times m})\ni \Phi \mapsto (C_\Phi(\gamma):
\gamma\in \Gamma_O),
\end{equation}
where $\Gamma_O$ is the set of unit-speed geodesics $\gamma:[0,\tau]\rightarrow M$ with $\gamma([0,\tau]) \subset O$ and both endpoints on $\partial M$, so called $O$-local geodesics.
In this article injectivity is refined to a H{\"o}lder type stability estimate; this estimate is our main result, precisely formulated in Theorem \ref{mainthm} below.
\newline
Non-abelian $X$-ray tomography provides the mathematical basis for the novel imaging technology of {\it polarimetric neutron tomography} \cite{nature1}\cite{nature2}, which seeks to determine a magnetic field  within a medium by probing it with neutron beams and measuring the spin-change that results from traversing the magnetic field. In this setting $\Phi$ takes values in  $\so(3)=\{A\in \R^{3\times 3}:A^T=-A\}$ and encodes the magnetic field and
%the magnetic field is encoded in an $\so(3)$-valued potential $\Phi$ and
 $C_\Phi(\gamma) \in SO(3)$ describes the resulting rotation of the spin-vector for a neutron travelling along $\gamma$. For a survey on further applications of non-abelian $X$-ray tomography we refer to \cite{Nov19}.\newline
Even in the simplest example, when $M$ is a Euclidean ball (thus geodesics are straight lines) and we have access to full data ($O=M$), the inverse problem described above is very challenging. It is nonlinear and for $m\ge 2$ no explicit inversion formula is known or expected to exist.
  \\
At the same time, real-life applications 
demand a computational approach to `solve' the inverse problem, typically in the presence of statistical noise on the measurements.
An attractive  and widely used such approach is Stuart's framework of {Bayesian inverse problems} \cite{DaSt17}, in which $\Phi$ is estimated from draws of a `posterior probability measure', which can be computed from a finite number of observations $C_\Phi(\gamma_1),\dots, C_\Phi(\gamma_n)$.\\ 
From a theoretical point of view this shifts the focus to a rigorous study of the performance of Bayesian algorithms. For the non-abelian $X$-ray transform this was initiated in \cite{MNP19}, where the authors prove a statistical consistency result in dimension $d=2$. This roughly asserts that potentials $\Phi$ can be recovered from $C_\Phi$ by a Bayesian algorithm with arbitrary accuracy, as the number of measurements $n \rightarrow \infty$. 
One of the key ingredients in the statistical analysis of non-linear inverse problems is a quantitative stability estimate with good control on the involved constants.  This principle has emerged in a series of recent papers, including \cite{MNP19} for the two-dimension non-abelian $X$-ray transform, as well as \cite{AbNi19} and \cite{GiNi19}, which analyse the Calder{\'o}n-problem and an inverse problem for the Schr{\"o}dinger equation respectively. In our case,  establishing consistency in $L^2$-norm requires a stability estimate of the form %The general form of such a stability estimate in our case is
\begin{equation}
\Vert \Phi - \Psi \Vert_{L^2} \le C(\Phi, \Psi) \cdot  d(C_\Phi, C_\Psi),
\end{equation}
\nopagebreak[1]where $C(\Phi,\Psi)>0$ is bounded over large classes of potentials $\Phi,\Psi$ and $d(\cdot, \cdot)$ 
 is an appropriate (semi-)metric. In \cite{MNP19}, the authors prove such an estimate in the two-dimensional case for $d(\cdot,\cdot)$ given by the distance in an $H^1$-Sobolev-space. Using an interpolation argument they derive further stability estimates with $d(\cdot,\cdot)=\Vert \cdot - \cdot \Vert^\mu_{L^2}$ and $\mu \in (0,1)$. Our main theorem contains a version of this H{\"o}lder type stability estimate for $d\ge 3$ and implies essentially the same consistency result as in two dimensions, however with the caveat of requiring priors of higher regularity and obtaining a slower rate of convergence. \newline
In a Euclidean setting the two-dimensional results from \cite{MNP19} are relevant also in higher dimensions, as one can reduce to $d=2$ by recovering $\Phi$ slice by slice; nevertheless there are several reasons to study the case $d\ge 3$ intrinsically.  Besides the applicability to a wider class of geometries,  partial data results become available, which for $d=2$ are less well understood and not available in general   \cite{Bo11}. This is of direct relevance to real-life applications, where one might have access only to localised measurement data. Further,  the methods for proving injectivity are very different in $d=2$ versus $d\ge 3$ and the quest for new stability estimates requires refining the methods to make them more quantitative,  which might in turn prove useful in other problems. This is especially true in $d \ge 3$, where injectivity is proved  by means of a novel and extremely versatile technique, as explained in the next paragraph.\newline
 The working horse behind many partial data results in $d\ge 3$, for non-abelian $X$-ray tomography as well as boundary rigidity and some other geometric inverse problems, is 
a ground-breaking technique of Uhlmann and Vasy \cite{UhVa16}. Their method automatically provides a local %l\footnote{By local we mean in a sufficiently small neighbourhood $O$ of a strictly convex boundary point.}  
stability estimate for the linearised problem; however, there are two less welcome features: The necessity of smooth data and the need to globalise. Let us elaborate on these points to explain the main technical contributions of this article.\\
With microlocal analysis at the core of the method, smoothness of the underlying data (in our case the potential $\Phi$) is not easily relaxed to lower regularity; in particular the constants in the local stability estimate a priori depend continuously on $\Phi$ only in the $C^\infty$-topology. However, statistical consistency demands better control and one of our main contributions is to show uniformity on \textit{arbitrarily large} $C^k$-balls (for $k\ge 0$ sufficiently large).\\ 
By `globalisation' we mean the extension of injectivity from small neighbourhoods of boundary points to larger domains, or all of $M$, via a layer stripping argument. As the initial domain of injectivity depends on the potentials $\Phi$, the layer stripping argument becomes more delicate and another contribution  of this paper is to carefully combine the arguments from \cite{PSUZ19} and \cite{SUV16} to globalise stability estimates for the non-abelian $X$-ray transform.
 
 \begin{comment} \newline
 The stability  analysis in \cite{MNP19} is valid  in  the case that $M$ is a simple\footnote{$(M,g)$ is called simple, if it is compact, non-trapping, free of conjugate points and has a strictly convex boundary. Examples are the Euclidean disk and small perturbations of it.  }%Examples of simple surfaces are the Euclidean disk and small perturbations of it.} 
  surface ($d=2$), one has access to full data ($O=M$)  and $\Phi$ takes values in $\u(m)$, the space of skew-hermitian matrices. In $d\ge 3$ it suffices to have  access only to partial data ($O\subset M$ with foliation condition) and general  potentials with values in $\C^{m\times m}$ are allowed.\\
In a Euclidean setting the two-dimensional result is satisfying also for higher dimensions, as one can reduce to $d=2$ by recovering $\Phi$ slice by slice. Besides the applicability to a wider class of geometries, the main motivation for studying the case $d\ge 3$ intrinsically is the availability of partial data results ($O\neq M$). This is of direct relevance to real-life applications, where one might have access only to localised measurement data. In dimension $d=2$ however, partial data result are less well understood and not available in general \cite{Bo11}.\\ 
The restriction to $\u(m)$-valued potentials in \cite{MNP19} is mainly due to the use of Pestov-identities, which are not available for general matrix-potentials. However, using the techniques from a recent  injectivity result \cite{PaSa20}, one might be able to extend  \cite{MNP19} to general potentials.\\
\end{comment}

\subsection{Notation and Background} 
We denote with $SM=\{(x,v)\in TM:\vert v\vert =1\}$ the unit-sphere bundle of $M$ and write $\pi:SM\rightarrow M$ for the projection onto the base variable. $SM$ is itself a manifold with boundary and, writing $\nu$ for the inward-pointing unit-normal to $\partial M$, we can decompose $\partial SM$ into 
\begin{equation*}
\partial_\pm 
SM=\{(x,v) \in SM: x\in \partial M, \pm \langle \nu(x), v\rangle \ge 0\}.% \quad, %\partial_0SM = \partial_-SM\cap \partial_+SM.
\end{equation*}
%which we call influx-boundary ($+$) and outflux-boundary ($-$) respectively.\\
Let $X$ be the geodesic vector field on $SM$ and $\varphi_t$ the geodesic flow. We then write $\gamma_{x,v}(t)=\pi(\varphi_t(x,v))$ for the geodesic adapted to $(x,v)\in SM$ and $\tau(x,v)\in [0,\infty]$ for the first time that $\gamma_{x,v}$ exits $M$. If $\tau(x,v)<\infty$ for all $(x,v)\in SM$, then $M$ is called \textit{non-trapping}. Further we say that $\partial M$ is {\it strictly convex} if  its second fundamental form is positive definite everywhere.\newline
If $M$ is non-trapping and has strictly convex boundary, then  $\partial_+SM$ naturally para\-me\-trises all geodesics with endpoints on $\partial M$ and  the non-abelian $X$-ray transform can be recast as map
\begin{equation}
C(M,\C^{m\times m})\rightarrow C(\partial_+SM,\Gl(m,\C)),\quad \Phi\mapsto C_\Phi.
\end{equation}
Precisely,  we set $C_\Phi = U_\Phi \vert_{\partial_+SM}$, where  $U_\Phi:SM\rightarrow \Gl(m,\C)$ denotes the unique continuous solution (differentiable along the geodesic flow) of %the boundary value problem
\begin{equation}\label{bvp}
(X+\Phi)U_\Phi = 0 \text{ on } SM\quad \text{ and }\quad U_\Phi = \id \text{ on } \partial_-SM.
\end{equation}
 %More generally, if $G\subset \Gl(m,\C)$ is a matrix Lie-group with Lie-algebra $\g$, one can show that 
% the  non-abelian $X$-ray transform maps $C(M,\g)\rightarrow C(\partial_+SM,G)$.\\
For  $O\subset M$ open we write $\M_O\subset \partial_+SM$ for the open set of all $(x,v)$ for which $\gamma_{x,v}(t)\in O$ for $0\le t\le \tau(x,v)$. The set $\M_O$ parametrises the collection $\Gamma_O$ of $O$-local geodesics. The following condition, introduced in this form in \cite{PSUZ19}, ensures that $O$ is scanned by sufficiently many geodesics emerging from $\M_O$ and allows to prove an injectivity result as stated below.

\begin{definition}\label{foliationcondition}
	An open subset $O\subset M$ satisfies the \textit{foliation condition}, if %$\partial M \cap O$ is strictly convex and 
	there is a smooth, strictly convex function $\rho: O \rightarrow \R$ which is \textit{exhausting} in the sense that  $O_{\ge c}=\{x \in O: \rho(x) \ge c\}\subset M$ is compact for all $c>\inf_O \rho$.
\end{definition}

\begin{theorem} \label{PSUZ}[Paternain, Salo, Uhlmann, Zhou, 2019 \cite{PSUZ19}]\,\\ Let $d\ge 3$,  assume that $\partial M$ is strictly convex and $O$ satisfies the foliation condition. Then for smooth potentials $\Phi,\Psi:M\rightarrow \C^{m\times m}$ we have that \begin{equation}
C_\Phi = C_\Psi \text{ on } \M_O \quad \Longrightarrow \quad \Phi = \Psi \text{ on } O.
\end{equation}
 %we have $\Phi = \Psi$ on $O$.\hfill \qed
\end{theorem}

In fact, the authors of \cite{PSUZ19} consider a more general situation, where  scattering data is defined with respect to attenuations $\A(x,v) = \Phi(x) + A_x(v)$ that may depend on the direction $v$  up to first order. That is, $\Phi$ is a matrix potential as above and $A\in \Omega^1(O,\C^{m\times m})$ is a matrix-valued one-form. In that case a similar result holds true, but one only has injectivity up to  the gauge $\A \mapsto u^{-1} \d u + u^{-1} \A u$ ($u: O\rightarrow \Gl(m,\C)$ smooth).\\
For the full-data problem $(O=M)$, the foliation condition reduces to the existence of a strictly convex function on $M$ and is set into relation with other geometric properties of $M$ in Section 2 of \cite{PSUZ19}: For example, if $M$ (with $\partial M$ strictly convex) supports a strictly convex function, it is automatically non-trapping and contractible. Conversely, if $M$ has non-negative sectional curvatures (or non-positive sectional curvatures and it is simply connected), then it admits a strictly convex function.\newline
Let us  conclude with a  brief overview of the history of the problem. Assuming a flat background geometry and access to full data, the problem was first studied by Vertgeim \cite{Ver91}, with further pioneering work by Novikov \cite{Nov02} and Eskin \cite{Esk04}, who established injectivity in dimension $d\ge 3$ and $d=2$ respectively (up to gauge in the general problem mentioned above).\\
In  the geometric setting and for $d=2$,  the full data problem is typically studied on compact surfaces $(M,g)$ that are {\it simple} in the sense that $\partial M$ is strictly convex and $M$ is assumed to be non-trapping and free of conjugate points.  There, injectivity of the map
\begin{equation}
C^\infty(M,\g)\rightarrow C^\infty(\partial_+SM,G), \Phi\mapsto C_\Phi
\end{equation}
(where $G\subset \Gl(m,\C)$ is a matrix Lie group with Lie-algebra $\g$) was first proved for $G = U(m)$ (the unitary group) in \cite{PSU12}.  In the case $G=\Gl(m,\C)$ injectivity was established under a negative curvature assumption in \cite{PaSa18} and, very recently, for general simple surfaces   \cite{PaSa20}.
Partial data results on the other hand (even for $m=1$) are less well understood in $d=2$  \cite{Bo11} and there is no analogue for \eqref{PSUZ} for smooth (non-analytic) potentials.

  \subsection{Main result} %Let  $(M,g)$ be compact, non-trapping and with strictly convex boundary. 
  Our main analytical result is the following stability estimate for the non-abelian $X$-ray transform on a compact manifold $(M,g)$, assumed to be non-trapping and have a strictly convex boundary.
%  \footnote{For partial data results ($O\neq M$) the geometry of $M$ outside of $O$ is irrelevant.  In this case assuming the global conditions is a matter of convenience.}  Our main analytical result is the following stability estimate:% for the non-abelian $X$-ray transform.

  \begin{theorem}\label{mainthm}
  	Suppose $d\ge 3$  and $K\subset O\subset M$ are such that $K$ is compact and $O$ is open and satisfies the foliation condition. Then for smooth potentials $\Phi,\Psi:M\rightarrow \C^{m\times m}$ %$\Phi, \Psi \in C^\infty(M, \C^{m\times m})$  %with $\supp\Phi, \supp\Psi \subset K$ 
  	we have
  	\begin{equation} \label{mainthm1}
  	\Vert \Phi - \Psi \Vert_{L^2(K)} \le C(\Phi,\Psi) \cdot \Vert C_\Phi - C_\Psi \Vert_{L^2(\M_O)}^{\mu(\Phi,\Psi)},
  	\end{equation}
  	where $C>0$ and $\mu \in (0,1)$ obey an estimate of the form
  	\begin{equation} \label{mainthm2}
  	C(\Phi,\Psi)\vee \mu(\Phi,\Psi)^{-1} \le \omega(\Vert \Phi \Vert_{C^k(M)} \vee \Vert \Psi \Vert_{C^k(M)})
  	\end{equation}
  	for some non-decreasing function $\omega:[0,\infty)\rightarrow [0,\infty)$ and an integer $k\ge 0$.
  \end{theorem}

 Here the Lebesgue spaces $L^2(K)$ and   $L^2(\M_O)$ (with codomain $\C^{m\times m}$ suppressed from the notation) are defined with respect to the natural Riemannian volume forms on the ambient manifolds $M$ and $\partial_+SM$. The space $C^k(M)$ consist of functions which are $k$-times continuously differentiable up to $\partial M$ and a choice 
 of continuous norm $\Vert \cdot\Vert_{C^k(M)}$ (defined with respect to some atlas) is assumed to be fixed throughout the discussion. Further, the notation $a\vee b$ is used  for the maximum of two quantities $a,b>0$. \\
Finally we mention that in the formulation of Theorem \ref{mainthm} as well as below we assume smoothness of the involved potentials $\Phi,\Psi$ only for convenience.  In all cases one can derive results for potentials of finite regularity $C^k$ (for $k\ge 0$ determined by \eqref{mainthm2} or a similar bound) by means of an approximation argument: If $\Phi \in C^k(M,\C^{m\times m})$  is approximated by a sequence of smooth potentials $(\Phi_n:n\ge 0)$ in $C^k$-norm, then $C_{\Phi_n}\rightarrow C_\Phi$ in $H^k(\partial_+SM)$ by Corollary \ref{cor_naxray} below.  As the $C^k$-norms are bounded along the sequence,  a bound as in \eqref{mainthm2} prevents the constants from blowing up, such that the stability estimate persists in the limit.
 \newline
  In view of the local stability estimates for the linearised problem in \cite[Theorem 1.3]{PSUZ19} and the situation in $d=2$ \cite[Corollary 1.4]{MNP19} 
  one might expect a stronger result, with the right hand side of \eqref{mainthm1} being replaced by a Lipschitz-type bound $\le C(\Phi,\Psi) \cdot \Vert C_\Phi - C_\Psi \Vert_{F(\M_O)}$ in terms of a suitable function-space $F$, say of Sobolev-regularity $H^1$ or even $H^{1/2}$. However, our result is both  in line with the available estimates for the related conformal boundary rigidity problem and  it is sufficient to prove statistical consistency. Let us elaborate on these points:\\
  The conformal boundary rigidity problem (determining a Riemannian metric $g$ on $M$ within a fixed conformal class from its boundary distance function) shares many features with the problem at hand: It is a gauge-free non-linear problem, which in dimension $d\ge 3$ is solved with Uhlmann-Vasy's method, also requiring a layer stripping argument to propagate injectivity into the interior of $M$. It is thus  natural to compare the available stability estimates \cite[Theorem 1.4]{SUV16} and indeed, equation (3) there is of a similar form as \eqref{mainthm1} here. Both here and there, the passage to the weaker H{\"o}lder-type estimate is an artefact of the globalisation procedure, which employs interpolation  at every step of the layer stripping argument.\\
To understand the statistical consequences of  Theorem \ref{mainthm}, we draw the comparison with the stability estimates in \cite{MNP19}. Their result concerns the full data problem on a simple surface $(M,g)$ and states  that
\begin{equation}\label{2dlipschitzstability}
\Vert \Phi - \Psi \Vert_{L^2(M)} \le C(\Phi,\Psi) \cdot \Vert C_\Phi - C_\Psi \Vert_{H^1(\partial_+SM)}
\end{equation}
for all smooth potentials $\Phi,\Psi:M\rightarrow \u(m)$ and some $C(\Phi,\Psi)>0$ which is bounded, as long as the $C^1$-norms of $\Phi$ and $\Psi$ are bounded.  This stability estimate is derived by means of a Pestov-type energy estimate which does not extend to higher dimensions and necessitates the restriction to $\u(m)$-valued potentials (although,
using the techniques from a recent  injectivity result \cite{PaSa20}, one might be able to extend it to general matrix potentials).  By means of the forward estimates in \cite{MNP19} and an interpolation argument (as described in the proof of Theorem 5.16 there),  estimate \eqref{2dlipschitzstability} can be brought into the following H{\"o}lder-type form, again valid for smooth $\Phi,\Psi:M\rightarrow \u(m)$
%When they eventually enter the consistency analysis, they take the following form, valid for smooth potentials $\Phi,\Psi:M\rightarrow \u(m)$:
   \begin{equation} \label{2dstability}
	\Vert \Phi - \Psi \Vert_{L^2(M)} \le C_k(\Phi, \Psi) \cdot \Vert C_\Phi - C_\Psi \Vert_{L^2(\partial_+SM)}^{(k-1)/k}, \quad k \ge 2
%	C_1 e^{C_2 (\Vert \Phi \Vert_{C^k(M)} \vee \Vert \Psi \Vert_{C^k(M)} )} \times \Vert C_\Phi - C_\Psi \Vert_{L^2(\partial_+SM)}^{(k-1)/k}, k \ge 2
	\end{equation}
	where 
	$ %\begin{equation} \label{2dstabilityu}
		C_k(\Phi,\Psi) = c_{1,k} \exp \left (c_{2,k} \left(\Vert \Phi \Vert_{C^k(M)} \vee \Vert \Phi \Vert_{C^k(M)}\right )\right)
	$ %\end{equation}
	for constants $c_{1,k},c_{2,k}>0$ only depending on $(M,g)$ and $m$. %Note that this differs from \cite[Corollary 2.3]{MNP19}, which asserts a Lipschitz-type $L^2$-$H^1$ stability estimate; the formulation above can be derived by means of their forward-estimates and the interpolation inequality ${\Vert \cdot \Vert_{H^1}}\lesssim_k \Vert \cdot \Vert_{L^2}^{(k-1)/k}{ \Vert \cdot \Vert_{H^k}^{1/k}}$ ($k \ge 2$) (cf. Lemma \ref{interpol}) as described in the proof of Theorem 5.16 there.  The two-dimensional result 
	This resembles the estimates given in Theorem \ref{mainthm} and indeed, in Section \ref{s_stats} we show that the statistical analysis of \cite{MNP19} carries over to the full data case ($O=M$) in $d\ge 3$: In a Bayesian framework, and under a suitable choice of priors $\Pi_n$ on $C(M,\R^{m\times m})$, %(depending on the regularity parameter $k$ in Theorem \ref{mainthm}), 
	the following consistency result holds true:

\begin{theorem}[Consistency] \label{consistencymain} Let $\Phi_0\in C^\infty(M,\R^{m\times m})$ and suppose  we observe $(X_i,V_i)$ and $Y_i =C_{\Phi_0}(X_i,V_i) + \epsilon_i$ ($i=1,\dots,n$), where the directions $(X_i,V_i)\in \partial_+SM$ are drawn uniformly at random and $\epsilon_i\in \R^{m\times m}$ is independent Gaussian noise. Then, as the sample size $n\rightarrow \infty$, the potential $\Phi_0$ can be recovered  as $L^2$-limit (in probability) of the posterior means $\mathbb{E}_{\Pi_n}[\Phi\vert (X_i,V_i,Y_i)_{i=1}^n] \in C(M,\R^{m\times m})$. %computed with respect to suitably chosen priors $\Pi_n$.
\end{theorem}

As the statistical analysis is conceptually independent of the remaining paper, a more detailed discussion of the underlying priors and a  comparison with \cite{MNP19} is postponed to Section \ref{s_stats}. The theorem above is restated in Theorem \ref{Bcont} and the remarks thereafter.\newline
Continuing our discussion of Theorem \ref{mainthm}, we remark that estimate \eqref{mainthm2} is a way of saying that for smooth potentials $\Phi$ and $\Psi$ lying inside of a fixed ball $\{{\Vert \cdot \Vert_{C^k(M)} }\,\le A\}$ ($A>0$), one may choose the constants $C$ and $\mu$ uniformly. This is a stronger result than the uniformity in \cite[Theorem 1.4]{SUV16} (conformal boundary rigidity), which only holds over sufficiently small balls. However, similar to the just cited result, the required regularity $k$ for which \eqref{mainthm2} is true, is unknown. This is in stark contrast with the  two-dimensional situation in \eqref{2dstability}, where one can freely choose $k\ge 2$.  To the knowledge of the author, the available techniques to reduce the required regularity to some smaller $k'\ll k$ (cf. \cite[Theorem 2b]{FSU07}, where $k'=2$), only yield uniformity for \textit{generic} elements of $C^{k'}$, which is not sufficient for the  statistical application mentioned above.

   \subsection{Key ideas and structure} \label{ss_keyideas}  Analysis of the non-abelian $X$-ray transform starts with a pseudo-linearisation identity  that we will now describe. Given a potential $\Phi\in C^\infty(M,\C^{m\times m})$ we call any (smooth) solution  $R:SM\rightarrow \Gl(m,\C)$ to $(X+\Phi)R=0$ on $SM$ an {\it integrating factor} for $\Phi$. Smooth integrating factors always exist in our setting ($M$ compact, non-trapping \& with strictly convex boundary) and can be used to express the non-abelian $X$-ray transform in  terms of the linear, weighted $X$-ray transform
   \begin{equation}\label{weighted}
   I_W f(x,v) = \int_0^{\tau(x,v)} Wf(\varphi_t(x,v)) \d t,\quad (x,v)\in \partial_+SM,
   \end{equation}
   defined for $W:SM\rightarrow \C^{m\times m}$ and $f:M\rightarrow \C^{m}$.  Precisely, we have:
   \begin{lemma}\label{lem_pseudolin} Let $\Phi,\Psi\in C^\infty(M,\C^{m\times m})$ and suppose that $R_\Phi$ and $R_\Psi$ are smooth integrating factors for $\Phi$ and $\Psi$ respectively. Then we have
   \begin{equation}   \label{pseudolin}
  C_\Phi - C_\Psi = R_\Phi\cdot  I_{\mathcal{W}_{\Phi,\Psi}} (\Phi - \Psi) \cdot \alpha^*R^{-1}_\Psi \quad \text{ on } \partial_+SM,
   \end{equation}
   where $\alpha(x,v)=\varphi_{\tau(x,v)}(x,v)$ is the scattering relation of $(M,g)$ and the weight $\mathcal{W}_{\Phi,\Psi}:SM\rightarrow \mathrm{End}(\C^{m\times m})$ is  defined pointwise by $\mathcal{W}_{\Phi,\Psi}A = R_\Phi^{-1} A R_\Psi$ for $A\in \C^{m\times m}$.
   \end{lemma}
   
 Note that the weighted $X$-ray transform in \eqref{pseudolin} is to be understood `one level higher', identifying $\C^{m\times m} \cong \C^{m'}$ and $\End(\C^{m\times m}) \cong \C^{m'\times m'}$ for $m'=m^2$.  
   
    \begin{proof}
	Let $F_\Phi$ be a first integral for $R_\Phi^{-1}\vert_{\partial_-SM}$ , that is $F_\Phi: SM\rightarrow \Gl(m,\C)$ solves $XF_\Phi=0$ on $SM$ and $F_\Phi = R_\Phi^{-1}$ on $\partial_-SM$. Then $U_\Phi = R_\Phi F_\Phi$ satisfies \eqref{bvp} and $C_\Phi = U_\Phi\vert_{\partial_+SM}$.  Using the corresponding notation for $\Psi$ we have
	\begin{equation}
		U_\Phi - U_\Psi = R_\Phi \cdot (F_\Phi F^{-1}_\Psi - R^{-1}_\Phi R_\Psi) \cdot F_\Psi,
	\end{equation}
	which, when restricted to $\partial_+SM$, yields \eqref{pseudolin}. To see this, note that  $G = F_\Phi F^{-1}_\Psi - R^{-1}_\Phi R_\Psi$ satisfies $XG  = - \mathcal{W}_{\Phi,\Psi}(\Phi-\Psi)$ on $SM$ and $G = 0$ on $\partial_-SM$. The fundamental theorem of calculus now implies that $G\vert_{\partial_+SM} = I_{\mathcal{W}_{\Phi,\Psi}}(\Phi - \Psi)$ and since further  $F_\Psi \vert_{\partial_+SM} = \alpha^* R_\Psi^{-1}$, the proof is complete.
   \end{proof}

We can now summarise the content of the subsequent sections and lay out the general strategy to prove the main results of this article.\\
In {section \ref{s_forward}}, we prove a forward estimate for the map $\Phi \mapsto R_\Phi$, which allows to translate stability estimates for the weighted $X$-ray transform into one for the non-abelian one. Further consequences are forward estimates for $\Phi \mapsto C_\Phi$, which are of interest in statistical applications. The techniques in this section  are similar to the ones in \cite{MNP19}, suitably adjusted to deal with dimension $d\ge 3$ and integrating factors taking values in the non-compact group $\Gl(m,\C)$.\\
{Section \ref{s_microlocal}} prepares the further  analysis  by proving a quantitative version of the microlocal technique (local inversion of scattering operators near elliptic points), introduced in the context of $X$-ray transforms by Uhlmann and Vasy \cite{UhVa16}. We give a self-contained proof, emphasising quantitative bounds on the involved constants.\\
In {section \ref{s_robust}} we start the stability analysis by considering the weighted $X$-ray transform $f\mapsto I_W f$.  By \cite[Thm.\,1.3]{PSUZ19}, if $W:SM\rightarrow \Gl(m,\C)$ is a smooth invertible weight, then every convex boundary point $p\in\partial M$ has a neighbourhood $O$ such that for $K\subset O$ compact we have %and $f\in L^2(O)$ with $\supp f \subset K$ we have
\begin{equation}
\Vert f \Vert_{L^2(K)} \lesssim_K C \cdot \Vert I_W f \Vert_{H^1(\M_O)}.
\end{equation}
Here both $C>0$ and the maximal size of $O$ (say, measured by the largest radius $h>0$ for which the ball $B(p,h)\subset O$) depend on $W$ and we will be concerned with understanding their behaviour as $W$ varies. Standard techniques imply that $C(W)$ and $h(W)$ depend continuously on $W$ in the $C^\infty$-topology, but this is not sufficient for our purposes. Using the quantitative analysis from Section \ref{s_microlocal}, we can upgrade this to uniformity as long as  $\Vert W \Vert_{C^k(SM)} \vee \Vert W^{-1} \Vert_{L^\infty(SM)}$ (for some $k\gg 1$) remains bounded.\\
In {section \ref{s_glob}} we use the local stability result from the previous section to successively derive further stability estimates. First, using a layer stripping argument similar to the one in \cite{PSUZ19}, we extend stability to arbitrary sets satisfying the foliation condition. Next, we use the pseudo-linearisation formula to translate this into a stability estimate for the non-abelian $X$-ray transform and finish the proof of our main theorem.\\
Finally, in {section \ref{s_stats}} we illustrate the strength of Theorem \ref{mainthm} by proving a statistical consistency result similar to the one in \cite{MNP19}. This section is mostly expository, as the extension to higher dimensions and general $\R^{m\times m}$-valued potentials   is fairly straightforward.

\subsection*{Acknowledgements} I would like to thank Gabriel Paternain and Richard Nickl for suggesting the project this article is based on, as well as for their support and guidance while working on it. Further thanks go to  Plamen Stefanov, Andr{\'a}s Vasy, Peter Hintz and Xi Chen who generously offered help and answered my numerous questions on microlocal analysis, as well as Jiren Zhu, whose thesis helped to clarify several aspects of the Uhlmann-Vasy method. 
 %Lastly, the humanity and leniency of a certain border guard at Frankfurt Airport had a great impact on my wellbeing and productivity during the lockdown and is hereby thankfully acknowledged.

	\section{Forward Estimates}\label{s_forward}
	
	In this section $(M,g)$ is a compact, non-trapping Riemannian manifold  with strictly convex boundary $\partial M$ and dimension $d\ge 2$.  Further, as it requires no additional effort, we work in a slightly more general setting and replace matrix potentials $\Phi:M\rightarrow \C^{m\times m}$ by attenuations $\A:SM\rightarrow \C^{m\times m}$.\\
	Recall that an integrating factor for $\A$ is a solution to the transport equation $(X+\A)R= 0$  on $SM$. The main result of this section then reads as follows:

	\begin{theorem}\label{thm_mainintfac} For every $\A\in C^\infty(SM,\C^{m\times m})$ there exists an integrating factor $R_\A\in C^\infty(SM,\Gl(m,\C))$ with \begin{equation*}
	 \Vert R_\A^{\pm 1} \Vert_{C^k(SM)} \le c_{1,k} \exp(c_{2,k} \Vert \A\Vert_{L^\infty(SM)}) \cdot (1+  \Vert \A \Vert_{C^k(SM)})^k, \quad k\ge 0
	 	 \end{equation*}
	 	 for constants $c_{1,k},c_{2,k}>0$ only depending on $M$ and $m$. If $\A$ takes values in $\u(m)$, the exponential factors can be dropped.
	\end{theorem}

	In order to define $R_\A$, we use a standard trick to avoid differentiability issues at the glancing region $S\partial M$:	We embed $M$ into the interior of a slightly larger manifold $M_1$ and  extend $\A$ smoothly to an attenuation $\A_1:SM_1\rightarrow \C^{m\times m}$ with compact support in $SM_1^{\interior}$. Then
	\begin{equation*}
	(X+\A_1) U = 0 \text{ on } SM_1\quad \text{ and } \quad U = \id \text{ on } \partial_-SM_1
	\end{equation*}
	has a unique solution $U_{\A_1}:SM_1\rightarrow \Gl(m,\C)$, which is constant $\equiv \id$ near $S\partial M_1$ and thus smooth on all of $SM_1$. Setting $R_\A = U_{\A_1} \vert_{SM}$ gives the desired integrating factor and the forward estimate above is  a consequence of the following result, applied to the larger manifold $M_1$.

	\begin{proposition}\label{prop_intfactor} Let $\A\in C^k(SM,\C^{m\times m})$ ($k\ge 0$) and suppose $U_\A\in C^k(SM,\Gl(m,\C))$ solves  $(X+\A)U = 0$ on $SM$ and $U = \id$ on $\partial_-SM$.

	\begin{enumerate}[label=(\roman*)]
		\item \label{prop_intfactori} Writing $\tau_\infty = \sup_{SM} \tau$, we have 
		$\Vert U_\A \Vert_{L^\infty(SM)} \le  m^{1/2} \exp(\tau_\infty \Vert \A \Vert_{L^\infty(SM)}).$
		\item \label{prop_intfactorii}If $\supp \A \subset K$ for a compact set $K\subset SM^{\interior}$, then 
		\begin{equation*}
		\Vert U_\A  \Vert_{C^k(SM)} \le c  e^{(2k+1)\tau_\infty \Vert \A\Vert_{L^\infty(SM)}}	
	(1+ \Vert \A \Vert_{C^k(SM)})^k , % \Vert U_\A\Vert_{L^\infty(SM)} \cdot \Vert \A\Vert_{C^k(SM)}^k
		\end{equation*}
		for  a constant $c=c(k,m,K,M)>0$.
		\item \label{prop_intfactoriii} The assertions remain true if $U_\A$ is replaced by its inverse $U_\A^{-1}$. Further, if $\A$ takes values in $\u(m)$, the exponential factors can be dropped.
	\end{enumerate}
	\end{proposition}	
	
	\begin{proof}[Proof of Theorem \ref{thm_mainintfac}]
	Following the construction outlined above, Proposition \ref{prop_intfactor} yields an estimate of $R_\A$ in terms of the norms $\Vert \A_1 \Vert_{C^k(SM_1)}$  and it remains to replace this by $
	\Vert \A \Vert_{C^k(SM)}$. Formally, this can be achieved by using Seeley's extension operator $E:C^\infty(SM,\C^{m\times m}) \rightarrow C^\infty(SM_1,\C^{m\times m})$ (Lemma \ref{seeley}). One can arrange (by multiplying with a fixed cut-off), that $\supp E\A\subset K$ for all $\A \in C^\infty(SM,\C^{m\times m})$ and a fixed $K\subset SM_1^{\interior}$. Then, as $E$ is continuous between the respective $C^k$-spaces, setting $\A_1= E \A$ allows to estimate $\Vert \A_1 \Vert_{C^k(SM_1)} \lesssim \Vert \A \Vert_{C^k(SM)}$ as desired.
	\end{proof}
	
	\subsection{Proof of Proposition \ref{prop_intfactor}}	
	We start by constructing suitable commuting frames, adapting \cite[Lemma 5.1]{MNP19} to arbitrary dimensions $d\ge 2$.
	
\begin{lemma} \label{commutingframe} Suppose $\Sigma\subset\partial_+SM\backslash S\partial M$ is open and $\{ P_1,\dots, P_{2d-2}\}$ is a commuting frame of $T\Sigma$. Then these vector fields can be extended smoothly to the open set $W_\Sigma=\{\varphi_t(x,v):(x,v)\in \Sigma, 0\le t\le \tau(x,v)\}\subset SM$ %the vector fields $P_1,\dots,P_{2d-2}$ can be extended smoothly 
to yield a commuting frame $\{X,P_1,\dots,P_{2d-2}\}$ of $TW_\Sigma$.
\begin{comment} there exist smooth vector fields $P_1,\dots,P_{2d-2}$ such that
	\begin{enumerate}[label=(\roman*)]
	\item $\{X,P_1,\dots,P_{2d-2}\}$ is a commuting frame %of $TSM\vert_W$
	\item $P_1,\dots,P_{2d-2}$ are tangential to $W_\Sigma \cap\partial_-SM$
	\item Given $k\in \Z_{\ge 0}$ and a multi-index $\alpha\in \Z_{\ge 0}^{2d-2}$ we have
	\begin{equation}
	X^kP_1^{\alpha_1}\dots P_{2d-2}^{\alpha_{2d-2}} u \in L^\infty(W)\quad\text{ for all } u\in C^\infty(SM).
	\end{equation}
	\end{enumerate}
\end{comment}
	\end{lemma}
	
	\begin{remark}
	The Lemma can be strengthened to allow $\Sigma\subset \partial_+SM$ with $\Sigma \cap S\partial M \neq \emptyset$. In that case the extended vector fields are continuous on $W_\Sigma$ and smooth on $W_\Sigma\backslash S\partial M$. (One can show that the map $\Phi$ below is a homeomorphism on $\Sigma\times \R$ and an immersion in $\Sigma\backslash S\partial M\times \R$. Since we do not use the stronger result, we omit the details.)
	\end{remark}
	
\begin{proof}
Let $(N,g)$ be a \textit{no return extension} of $M$ (cf. Lemma \ref{noreturn}) and denote the geodesic flow on $N$ also by $\varphi_t$. 
We claim that the map
\begin{equation*}\label{comframe1}
\Phi: \Sigma \times \R\rightarrow SN,\quad (x,v,t)\mapsto \varphi_t(x,v)
\end{equation*}
 is a diffeomorphism onto its image. Injectivity  follows immediately from the no-return property: If $\Phi(x,v,t)=\Phi(y,w,s)$, then  $\gamma_{x,v}$ enters $M$ both at times $0$ and $t-s$, which is impossible unless $(x,v,t)=(y,w,s)$. It remains to prove that $\Phi$ is an immersion, so let us compute its derivative at $(x,v,t)\in \Sigma\times \R$: For a tangent vector $\xi \oplus a \partial_t \in T_{(x,v)}\Sigma\oplus T_t\R$ we have 
 \[ 
\Phi_*(\xi\oplus a \partial_t) = \d \varphi_t(x,v)(\xi) + a X(\varphi_t(x,v))\in T_{\varphi_t(x,v)}SN.
\] If $\Phi_*(\xi\oplus a \partial_t) =0$, then the previous display implies $\xi + a X(x,v) = 0$ and as $X$ is transversal to $\Sigma$, we must have $a=0$ and $\xi=0$. Hence $\Phi$ is an immersion and the claim follows from invariance of domain.\\
Now extend the vector fields $ P_1,\dots, P_{2d-2}$ to  $t$-independent smooth vector fields  $\tilde P_1,
\dots,\tilde P_{2d-2}$ on  $\Sigma\times \R$. Then  $\{\partial_t,\tilde P_1,\dots,\tilde P_{2d-2}\}$ is a commuting frame on $\Sigma \times \R$ which pushes forward  along $\Phi$ to a commuting frame $\{X,P_1,\dots,P_{2d-2}\}$ on $\Phi(\Sigma\times \R)\subset SN$. Restricting to $W_\Sigma=\Phi(\Sigma\times \R)\cap S M$  finishes the proof.
\end{proof}

\begin{proof}[Proof of Proposition \ref{prop_intfactor}]
Let us first remark why \ref{prop_intfactoriii} holds true. The inverse $U_\A^{-1}$ satisfies the equation $XU_\A^{-1} - U_\A^{-1}\A=0 $ and forward estimates can be derived with the same arguments as for $U_\A$. Further, if $\A$ is $\u(m)$-valued, then $U_\A\in U(m)$, which is compact. In particular $\Vert U_\A\Vert_{L^\infty(SM)}$ can be bounded by an absolute constant and no exponentials arise below.\\

To prove part \ref{prop_intfactori}, fix $(x,v)\in SM$ and note that $U(t)=U_\A(\varphi_t(x,v))$ solves   \[\dot U + \A(\varphi_t(x,v)) U = 0 \text{ for } 0\le t \le\tau(x,v)\quad  \text{ and } \quad U(\tau(x,v))=\id.
 \] 
	Let $v(t)= \vert U(t) \vert_F^2$ (with $\vert \cdot \vert_F$ the Frobenius norm), then
	$
	\dot v(t) = 2 \langle \dot U(t),U(t)\rangle_F = 2 \langle - \A(\varphi_t(x,v)) U(t),U(t)\rangle_F \le 2 \vert \A(\varphi_t(x,v)) \vert_F \cdot \vert U(t) \vert^2_F,
	$
	where we used the Cauchy-Schwarz inequality and the sub-multiplicativity of the Frobenius-norm. Thus by Gronwall's inequality (with reversed time) we have
	\begin{equation*}
	v(t) \le v(\tau(x,v)) \exp\left(\int_t^{\tau(x,v)} 2 \vert \A(\varphi_s(x,v)) \vert_F \d s \right),\quad 0\le t \le \tau(x,v).
	\end{equation*}
	Choose $t=0$, such that the  left hand side becomes $\vert U_\A(x,v)\vert_F^2$. Note that $v(\tau(x,v))=\vert \id \vert_F^2=m$ and crudely bound the integral in the exponential by $2 \tau_\infty  \Vert \A \Vert_{L^\infty(M)}$. This concludes the proof of \ref{prop_intfactori}.\\
	
	In order to show \ref{prop_intfactorii}, we use the following inequality, which (in the unitary version) appears as part of Lemma 5.2 in \cite{MNP19}: If $\A,F:SM\rightarrow \C^{m\times m}$ are continuous and $G\in C(SM,\C^{m\times m})$ is the unique solution to $(X+\A)G=-F$ on $SM$ and $G=0$ on $\partial_-SM$, then
	\begin{equation}\label{pf_intfactor1}
	\Vert G \Vert_{L^\infty(SM)} \le m\tau_\infty \exp(2\tau_\infty \Vert \A \Vert_{L^\infty(SM)})\cdot \Vert F\Vert_{L^
	\infty(SM)}.
	\end{equation}
	We repeat its proof: One readily checks that 
		\begin{equation*}
			G(x,v) = -U_\A(x,v) \int_0^{\tau(x,v)} U_\A^{-1}F(\varphi_t(x,v)) \d t ,\quad (x,v)\in SM
		\end{equation*}
	and thus  $\Vert G\Vert_{L^\infty(SM)}\le \tau_\infty \Vert U_\A \Vert_{L^\infty(SM)} \Vert U_\A^{-1} \Vert_{L^\infty(SM)} \Vert F\Vert_{L^\infty(SM)}$. The norms of $U_\A^{\pm1}$ can be bounded with \ref{prop_intfactori} and thus \eqref{pf_intfactor1} follows.\\
	
	To proceed, take $\Sigma\subset \partial_+SM\backslash S\partial M$ a small open subset (such that it admits a commuting frame). Let $P_1,\dots,P_{2d-2}$ be the vector fields on $W_\Sigma$, as provided by Lemma \ref{commutingframe} and write $P^\alpha = P_1^{\alpha_1}\cdots P^{\alpha_{2d-2}}_{2d-2}$ for a multi-index $\alpha\in \Z^{2d-2}$. We claim that
	\begin{equation}\label{pf_intfactor2}
	\begin{split}
	\Vert U_\A \Vert_{k,\Sigma} &\overset{\mathrm{def}}{=}  \sup_{j+\vert \alpha\vert =  k}\Vert X^j P^\alpha U_\A  \Vert_{L^\infty(W_\Sigma)}
	\lesssim_{k,\Sigma}   e^{(2k+1)\tau_\infty \Vert \A\Vert_{L^\infty(SM)}}	
	\Vert \A \Vert_{C^k(SM)}^k
%\Vert U_\A\Vert_{L^\infty(SM)}  \Vert \A\Vert_{L^\infty(SM)}\Vert \A\Vert_{C^k(SM)}^k.
	\end{split}
	\end{equation}
	for all $k\in \Z_{\ge 0}$.
	Since finitely many such sets $\Sigma_1,\dots,\Sigma_n$ suffice to ensure $K\subset \bigcup_i W_{\Sigma_i}$, we have
	$
	\Vert U_\A\Vert_{C^k(SM)}\le \sum_i\sum_{\ell\le k} \Vert U_\A \Vert_{\ell,\Sigma_i} \lesssim_k 
	 e^{2(k+1)\tau_\infty \Vert \A\Vert_{L^\infty(SM)}}	
	(1+ \Vert \A \Vert_{C^k(SM)})^k
	$	
	 and \ref{prop_intfactorii} follows.\\
	We prove \eqref{pf_intfactor2} by induction over $k\in \Z_{\ge 0}$. The case $k=0$ follows from part \ref{prop_intfactori}, so let $k\ge 1$ and assume the result is true for $k-1$.  Consider $G =  X^jP^\alpha U_\A$  for an integer $j\ge 0$ and multi-index $\alpha$ such that  $j+\vert \alpha \vert = k$. We have
	\[
	(X+\A)G = [\A,X^jP^\alpha] U_\A \text{ on } SM\quad \text{ and }\quad G = 0 \text{ on } \partial_-SM,
	\]
	where $[\cdot,\cdot]$ denotes the commutator and the zero boundary values follow from $\A$ having compact support and thus $U_\A$ being constant near $\partial_-SM$. By \eqref{pf_intfactor1} we conclude that
	$
	\Vert X^jP^\alpha U_\A\Vert_{L^\infty(W_\Sigma)} \lesssim e^{2\tau_\infty\Vert\A\Vert_{L^\infty(SM)}} \cdot \Vert [\A,X^jP^\alpha] U_\A \Vert_{L^\infty(W_\Sigma)}
	$
	and since $ [\A,X^jP^\alpha]$ is a  differential operator on $SM$ of order $k-1$ and with continuous coefficients $\lesssim_{k} \Vert \A \Vert_{C^k(SM)}$, we have
	\begin{equation}
	\Vert X^jP^\alpha U_\A\Vert_{L^\infty(W_\Sigma)} \lesssim_{k} e^{2\tau_\infty\Vert\A\Vert_{L^\infty(SM)}} \cdot \Vert \A\Vert_{C^k(SM)} \cdot \Vert U_\A\Vert_{k-1,\Sigma}.
	\end{equation}
	Inequality \eqref{pf_intfactor2} follows from the induction hypothesis and we are done.
\end{proof}

\subsection{Consequences and further forward estimates} We first recall that the standard linear $X$-ray transform
\[
\I: C^\infty(SM)  \rightarrow  C^\infty(\partial_+SM ),\quad \I F(x,v) = \int_0^{\tau(x,v)} F(\varphi_t(x,v)) \d t,
\]
is continuous as map $H^k(SM)\rightarrow H^k(\partial_+SM)$ for all $k\ge 0$ \cite[Theorem 4.2.1]{Sha94}.\footnote{Alternatively one could start with a forward estimate for $\I$ with respect to different function spaces and obtain corresponding results for weighted and non-abelian $X$-ray transforms.} Independently of Theorem \ref{thm_mainintfac}, this yields the following:

\begin{corollary} \label{cor_wxray} Let $f\in C^\infty(M,\C^m)$ and $ W\in C^\infty(SM,\C^{m\times m})$. Then 
\begin{equation}
\Vert I_W f \Vert_{H^k(\partial_+SM)} \lesssim_k \Vert W \Vert_{C^k(SM)} \cdot \Vert f \Vert_{H^k(SM)}\quad k\ge 0.
\end{equation}
\end{corollary}

\begin{proof}
As $I_Wf = \I(Wf)$, this follows immediately from the $H^k$-continuity of $\I$ (in its straightforward extension to vector-valued functions) and the fact that pull-back by $\pi:SM\rightarrow M$ yields a bounded linear map $\pi^*:H^k(M)\rightarrow H^k(SM)$.
\end{proof}

Further, using Lemma \ref{lem_pseudolin} (pseudo-linearisation) and Theorem \ref{thm_mainintfac} we obtain the following forward-estimates for the non-abelian $X$-ray transform:

\begin{corollary}\label{cor_naxray}
Let $\Phi ,\Psi \in C^k(M,\C^{m\times m})$, then \[\Vert C_\Phi - C_\Psi \Vert_{H^k(\partial_+SM)} \le c_k(\Phi,\Psi) \cdot \Vert \Phi - \Psi \Vert_{H^k(M)},\quad k\ge 0,
\]
where \[c_k(\Phi,\Psi) = c_{1,k} \exp (c_{2,k} \Vert \Phi \Vert_{L^\infty(M)} + \Vert \Psi \Vert_{L^\infty(M)} ) \cdot (1+ \Vert \Phi \Vert_{C^k(M)}+  \Vert \Psi \Vert_{C^k(M)})^{2k}\]
for constants $c_{1,k},c_{2,k}$ only depending on $(M,g)$ and $m$. Further, if $\Phi,\Psi$ take values in $\u(m)$, the exponential factors can be dropped.
\end{corollary}

\begin{proof}
We use the pseudo-linearisation identity $C_\Phi - C_\Psi = R_\Phi\cdot  I_{\mathcal{W}_{\Phi,\Psi}} (\Phi - \Psi)\cdot \alpha^*R_\Psi^{-1}$ from Lemma \ref{lem_pseudolin} for the integrating factors provided by Theorem \ref{thm_mainintfac}. The integrating factor $R_\Psi$, acting via multiplication on $H^k(\partial_+SM)$, has operator norm $\le \Vert R_\Phi \Vert_{C^k(\partial_+SM)} \le \Vert R_\Phi \Vert_{C^k(SM)}$. A similar bound holds for $\alpha^*R_\Psi^{-1}$, as $\alpha$ is a diffeomorphism and thus, by Corollary \ref{cor_wxray}, we obtain
\begin{equation*}
\Vert C_\Phi - C_\Psi \Vert_{H^k(\partial_+SM)} \lesssim_k \Vert R_\Phi \Vert_{C^k(M)}  \Vert \mathcal{W}_{\Phi,\Psi} \Vert_{C^k(SM)}  \Vert \Phi - \Psi \Vert_{H^k(M)}  \Vert R_\Psi^{-1} \Vert_{C^k(SM)}.
\end{equation*}
As $  \Vert \mathcal{W}_{\Phi,\Psi} \Vert_{C^k(SM)}  \le \Vert R_\Psi \Vert_{C^k(SM)} \Vert R_\Psi \Vert_{C^k(SM)}$, the proof is finished by applying the estimates from Theorem \ref{thm_mainintfac}.
\end{proof}

\newpage

   \section{Local Inversion of Scattering Operators} \label{s_microlocal}
This section prepares the local stability estimate from Section \ref{s_robust} by proving a quantitative version of the microlocal argument that underlies Uhlmann and Vasy's method from \cite{UhVa16}.\\ Their argument relies on the following phenomenon: In the context of Melrose's scattering calculus, ellipticity of an operator near a boundary point yields local injectivity. More precisely, if $X$ is a manifold with boundary  and $A: C_c^\infty(X^\interior) \rightarrow C^\infty(X^\interior)$ is a (classical) {\it scattering} pseudodifferential operator ($\psi$do),  then the leading order behaviour at $\partial X$  is captured by its \textit{scattering principal symbol}, which is a smooth function $\sigma_\sc: {}^\sc T_{\partial X}^*X\rightarrow \C$, defined on the total space of the scattering cotangent bundle over $\partial X$. Ellipticity at $p\in \partial X$ then means that
\begin{equation} \label{sclowerbound}
\inf_{\zeta\in {}^\sc T_{p}^*X }\vert \sigma_{\sc}(p,\zeta) \vert > 0
\end{equation}
and implies the existence of a neighbourhood $O\subset X$ of $p$ for which
\begin{equation}
 \ker A \cap \{u \in L^2(X):  \supp(u) \subset O\} = 0.
\end{equation}
Together with the Fredholm property between appropriate function spaces this can be upgraded to a stability estimate for functions supported in $O$. The purpose of this section is to show that the size of $O$ as well as constants in a stability estimate can be controlled in terms of a lower bound on the scattering principal symbol and an upper bound on a fixed semi-norm of $A$.\newline
To formulate the theorem, let $X$ be a compact manifold with boundary, fix a boundary defining function $\rho:X\rightarrow [0,\infty)$ and write $B(\partial X, h) = \{x \in X: \rho(x) < h\}$. Then in terms of the locally convex spaces %The main theorem of this section is formulated in terms of the locally convex spaces
\begin{itemize}
\item $\Psi^{m,\ell}_\sc(X) = $ Fr{\'e}chet space of classical scattering $\psi$do's of order $(m,\ell)$
\item $H^{s,r}_\sc(X) =$ Hilbert space of Sobolev-functions of regularity $(s,r)$,
\end{itemize}
discussed in Subsection \ref{ss_sc} below, our result reads as follows:

\begin{theorem}[Local inversion of scattering operators]\label{thmini}
Let $ V \subset \partial X$ be open and $K\subset X$ compact  with  $K\cap \partial X \subset V$. Suppose $A\in \Psi^{m,\ell}_\sc(X)$ satisfies 
\begin{equation}\label{thminilambda}
\lambda(A) = \inf\{ \vert \sigma_\sc(A) (z,\zeta)\vert: z\in V: \zeta\in {}^\sc T_z^*X \}>0.
\end{equation}
\begin{enumerate}[label=(\roman*)]
	 \item \label{thminia} There are $h, C>0$ such that all functions $u \in L^2(X) $ with support contained in $K\cap B(\partial X, h)$ obey the estimate
\begin{equation} \label{thminieq}
\Vert u \Vert_{L^2(X) } \le C \Vert A u \Vert_{H_\sc^{-m,-(d+1 + 2\ell)/2}(X)}.
\end{equation}
\item \label{thminib} As $A$ varies, the constants $h(A)$ and $C(A)$ %depend continuously on $A$ and 
satisfy 
\begin{equation} \label{thminibeq}
C(A) \vee h(A)^{-1} \le \omega(\Vert A \Vert \vee \lambda(A)^{-1} )
\end{equation}
for a non-decreasing function $\omega:[0,\infty)\rightarrow [0,\infty)$ (of polynomial growth) and a continuous $\Psi_\sc^{m,\ell}$-semi-norm $\Vert \cdot \Vert $.
\end{enumerate}
\end{theorem}

The proof of Theorem \ref{thmini} can be sketched as follows: After localising to an $h$-neighbourhood of $V\cap K$ (where $A$ is elliptic), one constructs a parametrix $A^+$ for which the residuals $R_A=\id - A^+ A$ have $L^2$-operator norms of order $O(h)$, such that for $h\ll 1$ a local inverse of $A$ can be obtained by a Neumann series.  In order to derive a quantitative bound as in \ref{thminib} one then needs to find how certain operator norms of $A^+$ and $R_A$ depend on $A$.\\
From the usual construction of parametrises, it is clear that the maps $A\mapsto A^+$ and $A\mapsto R_A$ will be continuous in the appropriate Fr{\'e}chet-topologies, but as the maps are nonlinear, a bound as in \eqref{thminibeq} is not immediate.  However,  using finite order parametrises, one can make microlocal constructions more economic,  such that all quantities depend only on $\lambda(A)$ and fixed semi-norm of $A$ (corresponding to a fixed number of derivatives of its full symbol). \\
This reasoning seems to be part of the microlocal analysis folklore; yet the author is not aware of any reference for it, let alone in the setting of scattering pseudodifferential operators on manifolds.  The novelty and usefulness of Uhlmann-Vasy's argument thus warrant a careful analysis.\newline
Finally,  we remark that making the semi-norm $\Vert \cdot \Vert$ from \eqref{thminibeq} more explicit is possible, but requires to further open up the microlocal analysis machinery at the cost of obscuring the main argument. At the same time the added benefit is minimal, for in later applications $A$ will be constructed in terms of a certain weight function $W$ and the map $W\mapsto A$ is both costly (in the sense that many derivatives of $W$ need to be bounded in order to obtain control of $\Vert A \Vert$) and difficult to analyse quantitatively.

\subsection{The scattering calculus} \label{ss_sc} %Let $X$ be a compact manifold with boundary. 

We summarise some aspects of Melrose's scattering calculus  \cite{Mel94} with the purpose of fixing notation and  gathering the most relevant results at one place. See also \cite{Mel94}\cite{Vas}\cite[\textsection2]{UhVa16} and  \cite[\textsection 3.2]{SUV17}.\\
First some general notation. Denote with $\bar \R^d$ the {\it radial compactification} of $\R^d$, obtained by glueing $\R^d$ and $[0,\infty)\times S^{d-1}$ along the identification $x \mapsto (\vert x\vert^{-1}, \vert x\vert^{-1} x)$. More generally, given a vector bundle $E\rightarrow X$, one can radially compactify the fibres  to obtain a bundle $\bar E\rightarrow X$ \cite[\textsection 1]{Mel94}.
Further, we let $\dot C(X) = \bigcap_k \rho^k C^\infty(X)$ denote the space of functions which vanish to infinite order at $\partial X$  (similarly defined over $X\times X$) and note that  the natural inclusion $\R^d\subset \bar \R^d$ induces an isomorphism $\mathcal{S}(\R^d)\cong \dot C(\bar \R^d)$. \newline
We can now recall the definition of $\Psi^{m,\ell}_\sc(X)$, the space of classical scattering pseudodifferential operators on $X$.

\begin{definition}\label{defsc} A linear operator $A:\dot C^\infty(X)\rightarrow \dot C^\infty(X)$ is in $\Psi^{m,\ell}_\sc(X)$, if the following two conditions are satisfied:
\begin{enumerate}[label=(\roman*)]
	\item The Schwartz-kernel of $A$ is smooth away from the diagonal of $X\times X$ and vanishes to infinite order at the boundary.%satisfies $\varphi(z) A(z,z') \psi(z')\in \dot C^\infty(X_z\times X_{z'})$.
	\item In local coordinates $(x,y)=(x,y_1,\dots,y_{d-1})$ with $x\vert_{\partial X} = 0$ we have %and a compact set $K$ in the chart domain, we have
	\begin{equation} \label{defsc1}
	Au(x,y) = \int e^{i\xi\frac{x-x'}{x^2} + i\eta\cdot \frac{y-y'}{x}}   a(x,y,\xi,\eta) u(x',y') \d \xi \d\eta \frac{\d x' \d y'}{(x')^{d+1}}, \hspace{-1em}
	\end{equation}
	for all $u\in C^\infty(X)$ with compact support within the chart domain, where $a:(0,\infty)_x\times \R^{d-1}_y \times \R_\xi \times \R^{d-1}_\eta\rightarrow \C$ is smooth and  satisfies
	\begin{equation} \label{defsc2}
	x^{\ell} \langle (\xi,\eta) \rangle^{-m} a(x,y,\xi,\eta) \in C^\infty([0,\infty)_x\times \R^{d-1}_y \times \bar \R^d_{(\xi,\eta)}).
	\end{equation}
\end{enumerate}

\end{definition}
Note that we use the order convention from \cite{UhVa16}, that is, $\Psi_\sc^{m,\ell}(X) $ %\subset \Psi_\sc^{m',\ell'}(X)$ 
increases as $m$ and $\ell$ increase.
The definition above differs from the (equivalent) one given in \cite{UhVa16} in that it describes $A$ in terms of the local model $[0,\infty)_x\times \R^{d-1}_y$ for $X$ rather than in terms of $\bar \R^d$. The formulation here is for example used in \cite[Proof of Prop.\,4.2]{SUV17} and has the advantage that $(\xi,\eta)$ provide natural coordinates for the scattering cotangent bundle introduced below.\\
For the sake of completeness we mention here that \eqref{defsc2} could be replaced by the condition
\begin{equation}
\vert (x\partial_x)^k \partial_y^\alpha \partial_{(\xi,\eta)}^\beta a(x,y,\xi,\eta)
\vert \lesssim_{k,\alpha,\beta} x^{-\ell} \langle (\xi,\eta)\rangle^{m-\vert \beta \vert } ,\quad (k,\alpha),\beta \in \Z^{d}_{\ge 0}
\end{equation}
to obtain the larger class $\Psi^{m,\ell}_{\mathrm{scc}}(X)$ of (not necessarily classical) scattering $\psi$do's. The  advantage of using classical operators is that their principal symbols can be realised as functions, rather than as elements in a quotient space. In particular there is a natural way to measure their magnitude (in the sense of size of semi-norms), which is crucial for the quantitative aspect of Theorem \ref{thmini}. \\
Finally, we remark that $\Psi^{m,\ell}_\sc(X)$ has a natural Fr{\'e}chet-space structure in which 
a sequence of operators $A_n$ converges to $0$, iff the (weighted) symbols $a_n$ in \eqref{defsc2} converge to $0$ in the $C^\infty$-topology. In this topology $\Psi^{m,\ell}_\sc(X)\subset \Psi_\sc^{m',\ell'}(X)$ is a closed subspace whenever $m\le m'$ and $\ell\le \ell'$\footnote{The equivalent statement is false in $\Psi^{m,\ell}_\mathrm{scc}(X)$, classicality is needed.}.\newline
Let us briefly discuss some key aspects of the  scattering calculus:
The leading order behaviour of an operator $A\in \Psi^{m,\ell}_\sc(X)$ at $\partial X$ can be described in coordinates, where $A$ takes form \eqref{defsc1}, by
\begin{equation} \label{defsc3}
\sigma_\sc(A)(y,\xi,\eta) = x^{\ell} \langle (\xi,\eta) \rangle^{-m} a(x,y,\xi,\eta)\vert_{x=0},
\end{equation}
which makes sense in view of the stated regularity in \eqref{defsc2}. In order to understand $\sigma_\sc$ invariantly, one defines the so called {\it scattering cotangent bundle}\footnote{Formally, one checks that the one-forms $\frac{\d x}{x^2},\frac{\d y_1} {x} ,\dots, \frac{\d y_{d-1}}{x}$ (for local coordinates $(x,y)$ as in \eqref{defsc1}) span a locally free sheaf $\E$ over $C^\infty(X)$. The vector bundle ${}^\sc T^*X\rightarrow X$ is then defined fibre-wise by ${}^\sc T_p^*X =\E(X)/I_p\E(X)$, where $I_p\subset C^\infty(X)$ is the ideal of functions vanishing at $p$, and the smooth structure is chosen such that the natural map $\E(X) \rightarrow C^\infty(X,{}^\sc T^*X)$ is an isomorphism.}  ${}^\sc T^*X\rightarrow X$ with fibres having the following coordinate-description:
\begin{equation}\label{defsc4}
{}^\sc T_p^*X =\left\{ \xi \frac{\d x}{x^2} + \eta \cdot \frac{\d y}{x} \big\vert_p\right\} \equiv \R^d_{(\xi,\eta)}.
\end{equation}
Let ${}^\sc \bar T^*X \rightarrow X$ be the ball-bundle obtained by radially compactifying the fibres of ${}^\sc  T^*X$. Then under the identification indicated in \eqref{defsc4}, definition \eqref{defsc3} yields a smooth map  $\sigma_\sc(A):{}^\sc \bar T^*_{\partial X} X\rightarrow \C$,  defined on the total space of the pull-back of ${}^\sc \bar T^*X$  to $\partial X$. 
We call $\sigma_\sc(A)$ the {\it scattering principal symbol} of the operator $A$.
The principal symbol map $A\mapsto \sigma_\sc(A)$ fits into a split exact sequence of Fr{\'e}chet-spaces:
\begin{equation} \label{sc_ses}
0 \rightarrow \Psi_\sc^{m,\ell-1}(X) \hookrightarrow \Psi_\sc^{m,\ell}(X) \xrightarrow{\sigma_\sc} C^\infty({}^\sc \bar{T}_{\partial X}^*X) \rightarrow 0
\end{equation}
By this we mean that it is a split exact sequence of vector spaces,  with all involved maps being continuous; in particular there is a continuous right inverse $r: C^\infty({}^\sc \bar{T}_{\partial X}^*X)  \rightarrow \Psi^{m,\ell}_\sc(X)$ to $\sigma_\sc$.\\
The scattering principal symbol is also called `principal symbol at finite points' and can be complemented by $\sigma_p$, the `principal symbol at fibre infinity'. While the joint symbol $(\sigma_p,
\sigma_\sc)$ is needed, e.g. for regularity questions, for our purposes it suffices to keep track of the boundary behaviour.\\
Exactness of \eqref{sc_ses} is stated in  \cite[Prop.\,20]{Mel94} (where the scattering principal symbol is called `normal operator' and denoted $N_\sc$), while a continuous linear right split (also called quantisation map) is discussed below equation (5.30) in the same notes.
Finally we remark here that our definition of $\sigma_\sc$ differs from the one in \cite{UhVa16}, where the authors do not incorporate the pre-factor $\langle (\xi,\eta)\rangle^{-m}$ in \eqref{defsc3}, which implies that ellipticity is witnessed by a lower bound $\vert \sigma_\sc\vert \gtrsim \langle (\xi,\eta)\rangle^{m}$ rather than   $\vert \sigma_\sc \vert \gtrsim 1$ as in \eqref{sclowerbound}.\newline
Next, we note that the product of two scattering $\psi$do's is again a scattering $\psi$do. In fact  multiplication of operators yields a bilinear continuous map
\begin{equation}
\Psi^{m,\ell}_\sc(X)\times \Psi^{m',\ell'}_\sc(X)\rightarrow \Psi_\sc^{m+m',\ell+\ell'}(X)
\end{equation}
and the principal symbol behaves multiplicatively, that is
\begin{equation}\label{sigmamult}
\sigma_{\sc}(AB) = \sigma_\sc(A) \cdot \sigma_\sc(B).
\end{equation}
The continuity claim can be verified by keeping track of semi-norms, when proving that scattering operators are closed under multiplication and is a direct consequence of \cite[Prop.\,3.5]{Vas}; for \eqref{sigmamult} see also  \cite[eqn.\,(5.1) and  (5.14)]{Mel94}. \newline
 Finally, a natural scale of Hilbert-spaces that scattering operators act on, is provided by $H_\sc^{s,r}(X)$. On  $\bar \R^d$ (the radial compactification of $\R^d$) these spaces can be defined in terms of the standard Sobolev-space on $\R^d$ as \begin{equation}
H^{s,r}(\bar \R^d) = \langle z\rangle^{-r} H^s(\R^d_z).
\end{equation}
In general, $H_\sc^{s,r}(X)$ is defined by locally identifying $X$ with open subsets of $\bar \R^d$. For $s\ge 0$ they are related to the standard Sobolev-spaces $H^s(X)$ as follows:
\begin{equation}\label{sobolevcomparison}
\begin{cases}
H^s(X) \subset H^{s,r}_\sc(X) & \text{ for } r\le -\frac{d+1}2 \\
  H^{s,r}_\sc(X)\subset H^s(X) & \text{ for } r\ge -\frac{d+1}2 + 2 s
  \end{cases}
\end{equation}
An operator $A\in 
\Psi^{m,\ell}_\sc(X)$ then is continuous as map  $A:H^{s,r}_\sc(X)\rightarrow H^{s-m,r-\ell}_\sc(X)$ and indeed the inclusion
\begin{equation}\label{ctsaction}
\Psi_\sc^{m,\ell}(X) \hookrightarrow  \mathcal{B}(H^{s,r}_\sc(X), H^{s-m,r-\ell}_\sc(X))
\end{equation}
into the space of bounded linear operators is continuous.
The statements above are proved in \cite[Section 2]{UhVa16} and \cite[Section 3.8]{Vas}, modulo continuity of \eqref{ctsaction}, which follows from the open mapping theorem.
   
\subsection{Proof of Theorem \ref{thmini}}  As outlined above, we want to construct a local, finite order parametrix for the scattering operator $A$. On the level of principal symbols this corresponds to composition with the map $z\mapsto 1/z$, suitably cut off near zero.  We thus start with a lemma that provides norm bounds for this composition map.
To this end let $\varphi:\C\rightarrow [0,1]$ be a smooth function, vanishing near zero and constant $\equiv 1$ for $\vert z \vert\ge 1$.  Write $\varphi_t(z) = \varphi(z/t)$ for $t>0$ and define
\begin{equation}
\inv_t :C^\infty(M)\rightarrow C^\infty(M),  \quad u \mapsto \left(x\mapsto \frac{\varphi_t(u(x))}{u(x)}\right)
\end{equation}
on an arbitrary compact manifold $M$ (with or without boundary), which will later be taken equal to the total space of ${}^\sc\bar T_{\partial X}^*X$. Then
\begin{lemma}
The map $\inv_t:C^\infty(M)\rightarrow C^\infty(M)$ is continuous with respect to the $C^\infty$-topology; further, for every $k\in \Z_{\ge 0}$ there exists $C=C(k)>0$ with
\begin{equation*}
\Vert \inv_t(u) \Vert_{C^k(M)} \le C \left( 1 + 1/t\right)^{k+1} \left(1 + {\Vert u \Vert_{C^k(M)}}\right)^k,\quad u\in C^\infty(M),t>0.
\end{equation*}
\end{lemma}

\begin{proof} We prove more generally that composition with $\chi\in C_b^\infty(\C)$ (that is, $\chi:\C\rightarrow \C$ is smooth and all derivatives are bounded) is continuous as map $C^\infty(M)\rightarrow C^\infty(M)$ and we have
\begin{equation}\label{concatenation1}
\Vert \chi \circ u \Vert_{C^k(M)} \lesssim_k \Vert \chi \Vert_{C^k_b(\C)} \cdot \left( 1 + \Vert u \Vert_{C^k(M)} \right)^k
\end{equation}
such that the result follows from setting $\chi(z) = \varphi_t(z)/z$.  For simplicity, we only consider the case that $u$ is real valued and $\chi \in \C_b^\infty(\R,\R)$ (the complex case only requires notational changes) and assume that $M$ has empty boundary, noting that the general case can then be obtained by means of Seeley's extension theorem (Lemma \ref{seeley}).\\
Choose local coordinates $x^1,\dots,x^d$ and note that $\partial^\alpha (\chi \circ u)$  (for $\alpha \in \Z^d$ with $\vert \alpha \vert=k$) may be written as finite linear combination of terms of the form
	\begin{equation}
	P_{m,\beta}u\defn\left (\chi^{(m)}\circ u\right) \cdot \prod_{i=1}^m \partial^{\beta_i} u, \quad m \le k, \beta_i \in \Z_{\ge 0}^{d} \text{ with } \sum_{i=1}^m \vert \beta_i \vert = k.
	\end{equation}
	Let $K\subset M$ be a compact set inside the chart that supports $x_1,\dots,x_d$. Then%  since $\chi$ has bounded derivatives,
	\begin{equation}
	\Vert P_{m,\beta } u \Vert_{L^\infty(K)} \lesssim_{m,\beta}  \Vert \chi \Vert_{C_b^k(\C)}\cdot \Vert u \Vert_{C^k(M)}^m \quad \text{ for all } u \in C^\infty(M)
\end{equation}		
	and
	\begin{equation}
	\Vert P_{m,\beta} u_n - P_{m,\beta} u \Vert_{L^\infty(K)} \rightarrow 0 \quad \text{ when } u_n \rightarrow u \text{ in } C^\infty(M).
	\end{equation}
	Since $C^k(M)$ can be normed by the sum of $\Vert \partial^\alpha \cdot \Vert_{L^\infty(K)}$, where $\alpha$ runs through multi-indices in $\Z_{\ge 0}^d$ with $\vert \alpha \vert\le k$ and $K$ through compacts inside of chart domains, the previous two displays establish the desired result.
\end{proof}

Next, we construct a local, finite order parametrix for an operator $A\in \Psi_\sc^{m,\ell}(X)$ which is elliptic in an open set $U\subset {}^\sc \bar T^*_{\partial X}X$. By this we mean that
\begin{equation}\label{deflambda}
\lambda(A) = \inf_U \vert \sigma_\sc(A) \vert >0,
\end{equation}
which encompasses the definition in Theorem \ref{thmini}, where we have $U=\{(z,\zeta):z\in V,\zeta\in {}^\sc \bar T^*_zX\}$ for some open set $V\subset \partial X$. Then:% the following result holds:
\begin{lemma}\label{parametrix}
There exists a local parametrix $A^+\in \Psi_\sc^{-m,-\ell}(X)$ and a residual operator $R_A\in \Psi^{0,0}_\sc(X)$ with the following properties.
\begin{enumerate}[label=(\roman*)]
\item  \label{paramatrix1} We have $A^+A= \id - R_A$ and $\supp \sigma_\sc(R_A) \cap U = \emptyset $
\item \label{parametrix2} Given continuous semi-norms $\Vert \cdot \Vert$ and $\Vert \cdot \Vert'$ on $\Psi_\sc^{-m,-\ell}(X)$ and $\Psi^{0,0}_\sc(X)$ respectively, there exists a continuous semi-norm $\Vert \cdot \Vert''$  on $\Psi^{m,\ell}_\sc(X)$ and an integer $k\ge 0$ such that, as $A$ varies, we have
\begin{equation}
\Vert A^+ \Vert \vee \Vert R_A\Vert' \lesssim_{m,\ell} \left(1 + \lambda(A)^{-1} \right)^{k} \cdot (1 + \Vert A \Vert'' )^k.
\end{equation}
\end{enumerate}
\end{lemma}

\begin{proof}
Let $r:C^{\infty}({}^\sc\bar T^*_{\partial X}X)\rightarrow \Psi_\sc^{-m,-\ell}(X)$ be a continuous right split for the short exact symbol sequence in \eqref{sc_ses} and define
\begin{equation}
A^+ =r\left( \inv_{\lambda(A)}(\sigma_\sc(A))\right) \in  \Psi_\sc^{-m,-\ell}(X).
\end{equation}
Then, as $U\subset  \{(z,\zeta)\in {}^\sc\bar T^*_{\partial X}X: \sigma(A)(z,\zeta)\ge \lambda(A) \}$, we have
\begin{equation}
\sigma_\sc(A^+) = \inv_{\lambda(A)}(\sigma_\sc(A)) = \sigma_\sc(A)^{-1} \quad \text{ on } U.
\end{equation}
In particular, defining $R_A = \id   - A^+ A$, we see from the multiplicativity of principal symbols that $\sigma_\sc(R_A) = 0$ on $U$, such that \ref{paramatrix1} holds true.\\
Next, given a continuous semi-norm $\Vert \cdot \Vert$ on  $\Psi_\sc^{-m,-\ell}(X)$, as $r$ is a continuous {\it linear} map between Fr{\'e}chet spaces, there exists an integer $k\ge 0$ with
\begin{equation*}
\Vert A^+ \Vert \lesssim  \Vert \inv_{\lambda(A)}(\sigma_\sc(A)) \Vert_{C^k({}^\sc\bar T^*_{\partial X}X)} \lesssim (1+\lambda(A)^{-1})^{1+k}\cdot \Vert \sigma_{\sc}(A) \Vert_{C^k({}^\sc\bar T^*_{\partial X}X)}^k, 
\end{equation*} 
with implicit constants uniform in $A\in \Psi^{m,\ell}(X)$ with $\lambda(A)>0$ and where the second inequality follows from the preceding lemma. Finally, $\sigma_{\sc}$ itself is a continuous linear map and thus  $\Vert \sigma_{\sc}(A) \Vert_{C^k({}^\sc\bar T^*_{\partial X}X)}\lesssim \Vert A \Vert'$ for an appropriate semi-norm $\Vert \cdot \Vert''$. This completes the bounds on $A^+$.\\
In order to bound $\Vert R_A \Vert'$ (for a given semi-norm $\Vert \cdot \Vert'$ on $\Psi_\sc^{0,0}(X)$), we use that multiplication of scattering operators gives a continuous bilinear map, such that
\begin{equation}
\Vert R_A\Vert' \lesssim 1 + \Vert A \Vert'''\cdot \Vert  A^+ \Vert''''
\end{equation} 
for an appropriate choice of semi-norms on the right hand side. Combining this with the bounds on $A^+$ the proof is complete.
\end{proof}

We are now in a position to prove a slightly more general version of Theorem \eqref{thmini}, which does not require ellipticity in all fibre-directions. For this recall that have fixed a boundary definition function $\rho:X\rightarrow [0,\infty)$.
   
\begin{proposition} [Microlocal version of Theorem \ref{thmini}]  \label{microlocalversion} Let $\Gamma \subset U \subset {}^\sc \bar T_{\partial X}^*X$ subsets such that  $\Gamma$ is compact and $U$ is open.
	\begin{enumerate}[label=(\roman*)]
	\item  \label{microlocalversion1} 
	Let $A\in \Psi_\sc^{m,\ell}(X)$ with $\lambda(A)>0$, as defined in \eqref{deflambda}. Then there exist constants $h,C>0$ as well as a continuous semi-norm $\Vert \cdot \Vert_0 $ on $\Psi_\sc^{0,0}(X)$  with the following property:
	If $Q=q_1Q_2\in \Psi_\sc^{0,0}(X)$ is the product of a function $q_1\in C^\infty(X)$ and an operator $Q_2\in \Psi_\sc^{0,0}(X)$ such that
	\begin{equation}\label{microlocalinversion1a}
	\Vert \rho q_1 \Vert_{L^\infty(X)} \Vert Q_2\Vert_0 < h \quad \text{ and } \quad \supp \sigma_\sc(Q_2) \subset \Gamma,
	\end{equation}
	then for all $u\in L^2(X)$ we have
	\begin{equation} \label{microlocalinversion1b}
	\Vert u \Vert_{L^2(X)} \le C \Vert Q \Vert_0 \cdot \Vert A u \Vert_{H_\sc^{-m,-(d+1+2\ell)/2}(X)} +  2 \Vert (\id - Q) u \Vert_{L^2(X)}. \hspace{-2.5em}
	\end{equation}
	\item \label{microlocalversion2}
	As $A$ varies in the open set of operators with $\lambda(A)>0$, the constants $h(A)$ and $C(A)$ obey an estimate of the form
	\begin{equation}
	C(A) \vee h(A)^{-1} \le \left(1 + \lambda(A)^{-1} \right)^{k} \cdot (1 + \Vert A \Vert )^k
	\end{equation}
	a continuous semi-norm $\Vert \cdot \Vert$ on $\Psi^{m,\ell}_\sc(X)$ and an integer $k\ge 0$.
	\end{enumerate}
	\end{proposition}

	Let us first demonstrate how Theorem \ref{thmini} follows from this result:

	\begin{proof}[Proof of Theorem \ref{thmini}]
		We apply Proposition \ref{microlocalversion} with $U=\pi^{-1}(V)$ and $\Gamma=\pi^{-1}(K')$, where $K'\subset V$ is a compact set that contains $K\cap \partial X$ in its interior and $\pi:{}^\sc\bar T^*_{\partial X}X\rightarrow \partial X$ is the natural projection; we denote with $h'$ and $C'$ the constants from \ref{microlocalversion1}. Let  $Q=q_1q_2\in \Psi^{0,0}(X)$ be the product of two functions $q_1,q_2\in C^\infty(X)$ with 
		\begin{equation}
		1_{B(\partial X, h)} \le q_1 \le 1_{B(\partial X, 2h)} \quad \text{ and  }\quad 1_K\le q_2 \le 1_{V'},
		\end{equation} where $h$ remains to be chosen and   $V'$ is a neighbourhood of $K$ with $V'\cap \partial X \subset  K'$. Now let $h>0$ be such that
		\begin{equation}
		\Vert \rho q_1\Vert_{L^\infty(X)} \Vert q_2\Vert \le 2h \Vert q_2\Vert = h',
		\end{equation}
		then \eqref{microlocalinversion1a} is satisfied and we obtain \eqref{microlocalinversion1b}. Since $\{u:  \supp(u) \subset K \cap B(\partial X, h) \} \subset \ker(\id - Q)$, this concludes the proof.
	\end{proof}

	\begin{proof}[Proof of Proposition \ref{microlocalversion}] 
	Let $A^+$ and $R_A$ be as in Lemma \ref{parametrix}. We first estimate the operator norm of $QR_A$, acting on $L^2(X) = H_\sc^{0,-(d+1)/2}(X)$. To this end, we write $QR_A = (\rho q_1)\cdot (\rho^{-1} Q_2 R_A)$ and treat the two factors separately. To estimate the second factor, consider the bilinear continuous map
	\begin{equation}\label{keysequence}
	\Psi^{0,0}_{\sc,\Gamma}(X) \times \Psi^{0,0}_{\sc,\Lambda}(X) \rightarrow \Psi_\sc^{0,-1}(X) \xrightarrow{\times \rho^{-1}} \Psi_\sc^{0,0}(X)\subset \B(L^2(X)),
	\end{equation}
	where the involved spaces and maps are defined as follows: For $L \subset {}^\sc \bar T^\ast_{\partial X}X$ compact we write
$\Psi^{0,0}_{\sc,L}(X)$ for the closed subspace of operators $P\in \Psi_\sc^{0,0}(X)$ with $\supp \sigma_\sc(P) \subset L$; we let $\Lambda = {}^\sc \bar T^*_{\partial X} X\backslash U$, such that $R_A\in \Psi_{\sc,\Lambda}^{0,0}(X)$. Then the first map in 	\eqref{keysequence} is  multiplication, which takes values in $\Psi_\sc^{0,-1}(X)$ as $\Lambda \cap \Gamma = \emptyset$. %, and the second map is multiplication by $\rho^{-1}$.\\
Now $\rho^{-1} Q_2 R_A \in \B(L^2(X))$ is the image of $(Q_2,R_A)$ under the map \eqref{keysequence} and hence its operator norm is bounded by $\Vert Q_2 \Vert_0 \cdot \Vert R_A \Vert_0$ for a continuous semi-norm $\Vert \cdot \Vert_0$ on $\Psi^{0,0}(X)$.
Further, multiplication by $\rho q_1$ has operator norm $\le \Vert \rho q_1 \Vert_{L^\infty(X)}$.  Overall, we get
\begin{equation}
\Vert QR_A \Vert_{L^2(X)\rightarrow L^2(X)} \le \Vert \rho q_1 \Vert_{L^\infty(X)} \cdot \Vert Q_2 \Vert_0 \cdot \Vert R_A \Vert_0.
\end{equation}

Put $h=h (A)= \Vert R_A \Vert_0^{-1}/2$, then if $Q$ obeys \eqref{microlocalinversion1a}, the operator norm of $QR_A$ is bounded by $1/2$, which means that $\id - QR_A$ is invertible in $\B(L^2(X))$. Write  $N=(\id - QR_A)^{-1} \in \B(L^2(X))$ for the inverse, then %the result follows from the identity
\begin{equation}
u = NQA^+Au + N(\id - Q)u\quad \text{ for all } u \in L^2(X).
\end{equation}	
Now $\Vert N \Vert_{L^2(X)\rightarrow L^2(X)}\le 2$ and thus,  assuming without loss of generality that $\Vert \cdot \Vert_0$ dominates the $L^2(X)$-operator norm, we obtain \eqref{microlocalinversion1b} with $C(A)$ being twice the $H_\sc^{0,-\frac{d+1}{2}}(X)\rightarrow H_\sc^{-m,-\ell-\frac{d+1}{2}}(X)$ operator norm of $A^+$. Finally,  the bound in \ref{microlocalversion2} follows from the one in Lemma \ref{parametrix} and we are done.
	\end{proof}
	
	\subsection{Vector valued case}

Theorem \ref{thmini} works equally well for operators that act between sections of vector bundles. In this section we discuss the necessary changes in the case of trivial bundles (which is all we need in the sequel). \\ Let us write $A\in \Psi^{m,\ell}(X;\C^k)$ for $k\times k$-matrices of operators in $\Psi^{m,\ell}(X)$, understood to act between vector-valued functions %$H^{s,r}_\sc(X,\C^k)\rightarrow H^{s-m,r-\ell}_\sc(X,\C^k)$ 
in the obvious way. The scattering principal symbol is then a matrix-valued map
$
\sigma_\sc(A): {}^\sc \bar T^*_{\partial X}X \rightarrow \C^{k\times k}
$
and, using the notation $\vert M \vert = (M^*M)^{1/2} \in \C^{k\times k}$ for matrices $M \in \C^{k\times k}$, ellipticity of $A$ is witnessed by an inequality of the form
\begin{equation}
\vert \sigma_\sc(A)\vert > \lambda \quad \Leftrightarrow \quad \forall t \in \C^k: \langle \sigma_{\sc}(A) t , t \rangle > \lambda \vert t \vert^2.
\end{equation} 
Using this notation, Theorem \ref{thmini} holds true for $u \in L^2(X,\C^k)$ and is proved in the same way as the scalar case.

   \section{Local Stability of the Weighted X-ray transform}\label{s_robust}
This section is devoted to the proof of the following theorem, which is a more quantitative version of Theorem 1.3 in \cite{PSUZ19}.

\begin{comment}Before  formulating the theorem, we make a few remarks on the geometry near convex boundary points. Let $(M,g)$ be a compact Riemannian manifold  with boundary. Then strict convexity of $\partial M$ near a point $p\in \partial M$ is equivalent to the existence of a smooth function $\vartheta:M\rightarrow \R$ with
\begin{equation}\label{vartheta}
\begin{cases}
& \vartheta \text{ strictly convex near } p \\
&\vartheta(x) \le \vartheta(p) = 0\quad \text{and} \quad \vert \d \vartheta(x) \vert_g \le 1  \text{ for } x\in M.
\end{cases}
\end{equation}
Then the  set $O=\{\vartheta>-c\}$ $(c>0)$ contains all geodesic balls $B(p,r)$ of radius $r<c$ and, for $c$ sufficiently small, $O$ obeys the foliation condition.
\end{comment}

	\begin{theorem} \label{linearlocalstable}
		Let $(M,g)$ be a compact Riemannian manifold of dimension $d \ge 3$ and suppose $p\in \partial M$ is a point of strict convexity. Then there exists a smooth function $ \tilde x: M\rightarrow \R$, strictly convex near $p$ and  satisfying
		\begin{equation*}
		\tilde x \le 0 = \tilde x (p) \quad\text{and} \quad \vert \d \tilde x\vert_g \le 1,
		\end{equation*}
		such that for all smooth, invertible matrix weights $W:SM\rightarrow \Gl(m,\C)$ the following holds true:
		\begin{enumerate}[label=(\roman*)]
		\item \label{linearlocalstable1} There exist $C,h>0$ with the following property: For $0<c<h$ let $B=B(p,c/2)$ and $O=\{\tilde x >-c\}$, then 
		\begin{equation}\label{linearlocalstable1eqn}
		\Vert f \Vert_{L^2( B)} \le C \Vert I_W f \Vert_{H^1(\M_O)} \quad \text{for } f\in L^2(M). %\quad \text{if }\supp(f) \subset B(p,h),
		\end{equation}
		\item \label{linearlocalstable2} As $W$ varies,  the maps $W\mapsto C(W)$ and $W\mapsto h(W)$ obey
		\begin{equation}\label{linearlocalstable2eqn}
		h(W)^{-1}\vee C(W)\le \omega(\Vert W\Vert_{C^k(SM)} \vee \Vert W^{-1} \Vert_{L^\infty(SM)})
		\end{equation} 
		for some non-decreasing $\omega:[0,\infty)\rightarrow [0,\infty)$ and an integer $k\ge 1$.
		\item \label{linearlocalstable3} Under small perturbations of $M$ and $p$, in a sense made precise below, one can choose $\omega$ and $k$ to be constant.
		\end{enumerate}
	   \end{theorem}	
   
   Let us remark on a few aspects of the theorem: The bound $\vert \d \tilde x\vert_g \le 1 $ can always be achieved by scaling $\tilde x$ and is included as it ensures that the set $O=\{\tilde x>-c\}$ $(c>0)$ contains the geodesic ball
   $B(p,c/2)$.\\
   % in order to make quantitative assertions about the size of its super-level sets: Note that set $O=\{\vartheta>-c\}$ $(c>0)$ contains all geodesic balls $B(p,r)$ of radius $r<c$.\\
 In order to make the perturbation result from part \ref{linearlocalstable3} precise, assume that $M=\{\vartheta \le 0\}$ for a smooth function $\vartheta:M\rightarrow \R$ that is strictly convex in a neighbourhood $U$ of $p$. Then for $t>0$ small, also the boundary of the mani\-folds $M_t=\{\vartheta\le -t\}\subset M$ is strictly convex in $U$ and the Theorem applies to the weighted $X$-ray transform of $M_t$ (defined via integrals over the shorter geodesics with endpoints on $\partial M_t$). Then \ref{linearlocalstable3} means that estimate \eqref{linearlocalstable2eqn}  can be made uniform  for $t>0$ sufficiently small and  $q \in \partial M_t \cap U$ close to $p$.

\begin{remark} The compactness condition is non-essential and has only been included to simplify bound \eqref{linearlocalstable2eqn}. For non-compact $M$ and $h(W)$ replaced by $h(W)\vee h^*$ for a fixed upper bound $h^*>0$, the relevant sets from part \ref{linearlocalstable1} lie within a compact subset $L\subset M$ and \eqref{linearlocalstable2eqn} remains true after replacing the right hand side by $\omega(\Vert W \Vert_{C_L^k(SM)} \vee \Vert W^{-1} \Vert_{L^\infty(SM\vert_L)})$. Here the semi-norm $\Vert \cdot \Vert_{C^k_L(SM)}$ is defined in local coordinates by taking the supremum over  $L$ of derivatives up to order $k$.
  In particular, if $M$ can be embedded into a compact manifold $M'$, then $\Vert \cdot \Vert_{C_L^k(SM)} \lesssim \Vert \cdot\Vert_{C^k(SM')}$.
  \end{remark}
	
  \begin{proof}[Proof of Theorem \ref{linearlocalstable}] The proof essentially consists of a careful inspection of the Uhlmann-Vasy method, which is comprised of the following steps:
  \begin{enumerate}
  \item \label{uvstep1} In a neighbourhood of $p$, the normal operator $I_W^*I_W$ is modified to a `localised normal operator' $A^\chi_W$, defined over an auxiliary manifold $X$ with  $p\in O = M\cap X$.
  \item \label{uvstep2} The operator $A_W^\chi$ is shown to lie in the class $\Psi^{-1,0}_\sc(X)$ (Definition \ref{defsc}), elliptic near the `artificial boundary' $\partial X$. By Theorem \ref{thmini}, it is thus locally invertible in a neighbourhood of $\partial X$.
  \item \label{uvstep3} A posteriori, the auxiliary manifold $X$ is  chosen such that the domain of injectivity includes $O=M\cap X$. Stability estimates for $A^\chi_W$ can then be translated into ones for $I_W$.
  \end{enumerate}
   
Using Theorem \ref{thmini}\ref{thminia}, the constants in the resulting stability estimate are then uniform under some control on $\sigma_\sc(A^\chi_W)$ and $\Vert A^\chi_W \Vert$  (for a semi-norm $\Vert \cdot \Vert$ on $\Psi^{-1,0}_\sc(X)$). As $A^\chi_W$ depends homogeneously and  (in the $C^\infty$-topology) continuously on $W$, this easily translates to uniformity in terms of $W$ and eventually yields \ref{linearlocalstable2}.\\ 
   We will now discuss the three steps above in more detail. However, as the method has been used in several previous articles (e.g. \cite{UhVa16}\cite{SUV16}\cite{SUV17}\cite{PSUZ19}), the exposition below will be brief and focus on the application of our quantitative result from the previous section.\newline
{\bf Step (1).}
	We embed $M$ into a closed manifold $(N,g)$ of the same dimension and extend the weight smoothly to $W:SN\rightarrow \C^{m\times m}$.
	As $p\in\partial M$ is a point of strict convexity, it admits a neighbourhood $U\subset N$  and coordinates $(\tilde x,y):U\xrightarrow{\sim} \R\times \R^{d-1}$ for which
	\begin{equation}
	\{\tilde x \ge 0\}\cap M = \{p\} \quad \text{and} \quad \tilde x \text{ is strictly convex near  } p
	\end{equation}
	(cf. \cite[Section 3]{PSUZ19} for a construction).
	The following constructions are carried out with respect  to a small parameter $0<c<c_0$ (and $c_0$ chosen later), noting dependencies when necessary. Change coordinates to $(x,y)=(\tilde x + c,y)$, such that $\{x \ge 0 \}$ is the intersection of $U$ with a compact manifold $X\subset N$ with {\it strictly concave} boundary near $p$. Consider the parametrisation
	\begin{equation}\label{projection}
	\R_x\times \R^{d-1}_y \times \R_\lambda \times S^{d-2}_\omega \rightarrow SU ,\quad (x,y,\lambda,\omega) \mapsto \frac{\lambda \partial_x+ \omega\partial_y}{\vert \lambda \partial_x+ \omega\partial_y \vert_g},
	\end{equation}
	with vectors parallel to $\partial_x$ missing in the image (they are negligible, as eventually we are interested in geodesics that are `nearly tangent' to $\partial X$) . Pulling back the geodesic flow via \eqref{projection} yields integral curves

	\begin{equation}\label{integralcurve}
	\gamma_{x,y,\lambda,\omega}(t) = \left( \gamma_{x,y,\lambda,\omega}^{(1)}(t), \gamma_{x,y,\lambda,\omega}^{(2)}(t) \right) \in \R\times \R^{d-1}
	% \quad \leadsto \quad \varphi_t(z,v)  \in \R_x\times \R^{d-1}_y
	\end{equation}
	and one may consider the following `localised normal operators', acting on smooth functions $f:[0,\infty)_x\times \R^{d-1}_y \rightarrow \C^m$ with suitable decay at $x=0$:
	\begin{equation}\label{defnormal}
	\begin{split}
	\hspace{-.475 em} A^\chi_Wf(x,y) = x^{-2} e^{-1/x} \iiint & W^*(x,y,\lambda,\omega) (Wf) \left( \gamma_{x,y,\lambda,\omega}(t), \dot \gamma_{x,y,\lambda,\omega}(t) \right)\\
	& e^{1/\gamma^{(1)}_{x,y,\lambda,\omega}(t)} \chi(x,y,\omega,\lambda/x) ~ \d t \d \lambda \d \omega.
	\end{split}
	\end{equation}
	This corresponds to equation (4.1) in \cite{PSUZ19}. Let us discuss the  ingredients of \eqref{defnormal} in detail: 	Without loss of generality we may assume that the interior of the box $B= [0,2c]_x \times [-1,1]^{d-1}_y$  contains the portion of $M$ within $U\cap X$. Further, the `localising function'  $\chi$ is assumed to satisfy\footnote{ Note that in \cite{PSUZ19}, the authors write $\chi=\chi(\lambda/x)$, suppressing the dependency on $(x,y,\omega)$.
	}
	\begin{equation}\label{sptcon}
	%X = \{ \chi \in C^\infty(\R^d_{(x,y)}\times S^{d-2}_\omega, \S(\R_s)) : \supp \chi \subset B \times S^{d-2} \times \R  \} 
	\supp \chi \subset B\times S^{d-2} \times [-C_0,C_0]
\end{equation}		
	for some $C_0>0$ and will later be chosen such that $A^\chi_W$ is elliptic in an appropriate sense. The domain of integration in \eqref{defnormal} is $[-\delta_0,\delta_0]_t \times \R_\lambda \times S^{d-2}_\omega$, where $\delta_0>0$ is chosen small enough to satisfy the following criteria: First we ask that the curves \eqref{integralcurve}, starting from $B$,  do not leave the coordinate chart for $\vert t \vert \le \delta_0$. Second, and after decreasing $c_0$ if necessary, we ask that 
	\begin{equation} \label{sptcon2}
	\gamma_{x,y,\lambda,\omega}^{(1)}(t) \ge \frac{C_1}{2} \left( t + \frac{\lambda}C_1\right)^2 + \left(x - \frac{\lambda^2}{2C_1} \right),\quad  (x,y)\in B, \vert t\vert, \vert \lambda \vert < \delta_0,
	\end{equation}
	for some $C_1>0$. (See equation (3.2) in \cite{UhVa16}, where this inequality is derived for $C_1$ essentially being a lower bound of the Hessian of $\tilde x$ near $p$).\newline
	\textbf{Step (2).}
	Note that $A_W^\chi$ may be viewed as operator $C_c^\infty(X^\interior,\C^m) \rightarrow C^\infty(X^\interior,\C^m)$ with Schwartz-kernel compactly contained in $(U\cap X)^2$. The crux is now that $A_W^\chi$ fits into Melrose' scattering calculus in the sense that $A_W^\chi \in \Psi^{-1,0}_\sc(X)$ and, upon a judicious choice of localiser $\chi$, is elliptic near $\partial X\cap M$. In particular, Theorem \ref{thmini} (local inversion of scattering operators) can be used.\\ 
	In order to give a precise statement, we recall that the constructions above depend on a parameter $c>0$ and there is a whole family of operators $A^\chi_W(c)$, defined over sub-manifolds $X_c\subset N$ (with $X_c\cap U=\{\tilde x + c \ge 0\}$). We may assume that there is a flow $\psi_c$ on $N$, defined for small $c>0$, for which $X_c=\psi_c(X_0)$. 
	
	\begin{theorem}\label{voodoo} Upon choosing $c_0,\lambda_0>0$ sufficiently small, we have:
	
	\begin{enumerate}[label=(\roman*)]
		\item \label{voodoo1} For all smooth localisers $\chi$  with \eqref{sptcon}, the  operator  $A^{\chi}_W(c)\in \Psi^{-1,0}_\sc(X_c)$. Further,  allowing $\chi$ to depend continuously on $c$, the map 
		\begin{equation} \label{mapcont}
		\begin{split}
	[0,c_0) \times C^\infty(SN,\C^{m\times m}) &\rightarrow  \Psi^{-1,0}_\sc(X_0), \quad (c,W) \mapsto 
	\psi_c^* A^{\chi_c}_W(c)
	  \hspace{-2em}	
	\end{split}
	\end{equation}
is continuous with respect to the natural Fr{\'e}chet-topologies. Moreover, for any continuous semi-norm $\Vert \cdot \Vert $ of $\Psi^{-1,0}_\sc(X_0)$ there is an integer $k\ge 0$ such that
\begin{equation}
\Vert \psi_c^* A^{\chi_c}_W(c) \Vert \lesssim \Vert W \Vert_{C^k(SN)}^2\quad \text{ for all } 0\le c < c_0. \label{upgrade}
\end{equation}
	\item \label{voodoo2} There exists a localiser $\chi$, smooth, satisfying \eqref{sptcon} and depending continuously on $c$, such that for all  $(c,W)$  in \eqref{mapcont} we have
	\begin{equation} \label{symbollowerbound}
	\vert \sigma_\sc(A^{\chi_c}_W(c)(z,\zeta)) \vert \ge \lambda_0 \Vert W^{-1} \Vert_{L^\infty(SN)}^{-2},\quad 
	z\in \partial X_c\cap M, \zeta \in {}^\sc T^*_zX_c. \hspace{-2em}
	\end{equation} 
	
		\end{enumerate}
	\end{theorem}
	
	\begin{proof}[Sketch of Proof] The proof is essentially carried out in Section 4 of \cite{PSUZ19}. We sketch the main aspects, highlighting dependencies on the weight.\\
	Either by first computing the Schwartz-kernel (\cite[Lem.\,4.1.]{PSUZ19}) or directly (akin to \cite{SUV17}), one verifies that $A_W^\chi$ has an oscillatory integral expression of the form \eqref{defsc1} and the pseudodifferential-property as well as the continuous dependency can be checked directly. We note here that continuous dependence on $c$ is already implicitly used in \cite{MNP19} and continuous dependence on $W$ is akin to continuous dependence on the metric as stated in e.g. \cite[Prop.\,4.2]{SUV17}.\\
	Further, \eqref{upgrade} can be derived from \eqref{mapcont} and the homogeneity of $A_W^{\chi_c}(c)$ in $W$. This can be seen easiest in a general functional analytic setting, where we are given two Fr{\'e}chet-spaces $E$ and $F$ and a continuous map
	\begin{equation}
	\varphi: [0,\infty)\times E\rightarrow F,\quad \text{with } \varphi(c,t \cdot )=t^2 \varphi(c,\cdot) \quad (t,c\ge 0).
	\end{equation}
	Then the collection of sets $\{(t,w):0\le t < \epsilon: \Vert w \Vert' <1 \}$, where $\epsilon>0$  and $\Vert \cdot \Vert'$ runs through continuous semi-norms of $E$, constitute a basis for the neighbourhoods of $(0,0)\in [0,\infty)\times E$. Given a continuous semi-norm $\Vert \cdot \Vert$ on $F$, the set $\{(c,w):\Vert \varphi(c,w) \Vert < 1 \}$  is an open neighbourhood of $(0,0)$ and thus we can find $c_0>0$ and $\Vert \cdot \Vert'$ with 
	\begin{equation}
	\{(c,w): 0\le c < c_0,  \Vert w \Vert' < 1\} \subset \{(c,w):\Vert \varphi(c,w) \Vert < 1 \}.
	\end{equation}
	Now take $0\le c <c_0$ and $w\in E$, then $(c,w/(2\Vert w \Vert'))$ lies in the left set and thus 
	\begin{equation}
	\Vert \varphi(c,w) \Vert= 4 (\Vert w \Vert')^2 \cdot\Vert \varphi(c,w/(2\Vert w \Vert')) \Vert\le 4 (\Vert w \Vert')^2,
	\end{equation}
	as desired.\\
	To prove \ref{voodoo2}, we first fix $c$. Then  the symbol in said oscillatory integral expression, restricted to $x=0$, takes the form
	\begin{equation} \label{fullsymbol}
	\begin{split}
	a(0,y,\xi,\eta) = \iiint &e^{i\xi (\hat \lambda \hat t +\alpha(0,y,0,\omega)\hat t^2 + i\eta \cdot \omega \hat t } \cdot e^{-{\hat \lambda \hat t  - \alpha(0,y,0, \omega) \hat t^2}} \\
	&  \times W^*W(0,y,0,\omega) \chi(0,y,\hat \lambda, \omega)  \d \hat \lambda \d \hat t \d \omega,
	\end{split}
	\end{equation}
	where $\alpha(x,y,\lambda,\omega)=(\d/\d t)^2 \gamma_{x,y,\lambda,\omega}^{(1)}(t) > 0$ (say, for $(x,y)\in B$) and the integral domain is $\R_{\hat \lambda}\times\R_{\hat t}\times S^{d-2}_\omega$. For the particular  choice $\chi(x,y,\hat \lambda, \omega) = \exp(-\hat \lambda^2/(2\alpha(x,y,\lambda,\omega)))$ (multiplied with a cut-off in $(x,y)$ to ensure that it is supported in $B$), the integral in the last display can further be evaluated to obtain a non-zero multiple of
	\begin{equation}
	\langle \xi \rangle^{-1} \int_{S^{d-2}} (W^*W)(0,y,0,\omega) e^{-\vert \eta \cdot \omega / \langle \xi \rangle \vert^2/ 2 \alpha(0,y,0,\omega)} \d \omega,
	\end{equation}
	which corresponds  to the second display below equation (4.10) in \cite{PSUZ19}. Following the reasoning of \cite[proof of Prop.\,4.3]{PSUZ19} below said expression yields
	\begin{equation}
	\langle (\xi,\eta) \rangle a(0,y,\xi,\eta) \ge 2\lambda_0\cdot \Vert W^{-1} \Vert_{L^\infty(SN)}^{-2}
	\end{equation}
	for a constant $\lambda_0$ only depending on the local geometry near $p$. Here it was used that $W^*W(0,y,0,\omega)$ is bounded from below by the square of the smallest singular value of $W$, which is in turn lower-bounded by $\Vert W^{-1} \Vert_{L^\infty(SN)}^{-2}$.\\
	The localiser $\chi$ above has full support in $\hat \lambda$ and thus fails to satisfy  \eqref{sptcon}. In the proof of \cite[Prop.4.3]{PSUZ19} $\chi$ is thus approximated by localisers with compact $\hat \lambda$-support, thus obeying \eqref{sptcon} for some $C_0>0$. From \eqref{fullsymbol} it follows that the approximation is uniform in $W$, at least under an a priori bound $\Vert W \Vert_{L^\infty(SN)} \le 1$. This proves part \ref{voodoo2} for all $W$ with $\Vert W \Vert_{L^\infty(SN)} \le 1$ and the general case follows from a scaling argument, noting that both sides of \eqref{symbollowerbound} are homogeneous in $W$  of degree $2$.\\
	Finally we comment on the $c$-dependency: Note that $\alpha(x,y,\lambda,\omega)$ (and thus $\chi$) implicitly depends on $c$ through the choice of $x=\tilde x  + c$. However, the dependence is clearly  continuous and $\alpha$ can be bounded in terms of the geometry near $p$. In particular, the bound \eqref{symbollowerbound} is uniform in $c$.
	\end{proof}

	By Theorem \ref{voodoo}, for an invertible weight $W$, the operator $A^{\chi}_W(c)$ is  locally elliptic for suitably chosen $\chi$ and sufficiently small $c>0$. In particular  Theorem \ref{thmini} can be applied and, for constants $C,h>0$ (depending on $W$ and $c$), we obtain
\begin{equation}\label{iniapp}
\Vert f \Vert_{L^2(X)} \le C \Vert A^\chi_W(c) f\Vert_{H_\sc^{1,-(d+1)/2}(X_c)},\quad \text{ if } \supp f \subset M\cap B(\partial X_c, h).
\end{equation}
Due to \eqref{upgrade} and \eqref{symbollowerbound}, the uniformity statement of Theorem \ref{thmini}\ref{thminib} gives
\begin{equation}\label{iniapp2}
C(W,c) \vee h(W,c)^{-1} \le \omega(\Vert W \Vert_{C^k(SN)} \vee \Vert W^{-1} \Vert_{L^\infty(SN)} ),
\end{equation}
valid for sufficiently small $c>0$ and all smooth weights $W:SN\rightarrow \Gl(m,\C)$. Here $\omega:[0,\infty)\rightarrow [0,\infty)$ is a non-decreasing function and $k\ge 1$. \newline
\textbf{Step (3).} From now on we argue with a fixed weight $W:SN\rightarrow \Gl(m,\C)$ and keep track of how the arising constants depend on 
\begin{equation}
A = \Vert W \Vert_{C^k(SN)} \vee \Vert W^{-1} \Vert_{L^\infty(SN)}>0.
\end{equation}
Fix $0<c<h$, then  $O_c=M\cap X_c$ lies in $M\cap B(\partial X_c,h)$ and consequently (dropping the $c$-subscripts from now on)
\begin{equation} \label{stabby}
\Vert  f \Vert_{L^2(O)} \lesssim_A \Vert A^\chi_W f\Vert_{H^{1,-(d+1)/2}(X)},\quad f\in L^2(M),
\end{equation}
where it is understood that $f$ is extended by zero  outside of $M$. In order to translate this into a stability estimate for $I_W$, we factor the operator $A_W^\chi$ as 
\begin{equation}\label{factor}
A_W^\chi f = x^{-2} e^{-1/x} L_W^\mu \tilde I_W (e^{1/x}f),\quad f\in L^2(M)
\end{equation}
where the operators $L^\mu_W$ and $\tilde I_W$ are defined as follows: We may assume that $\bar U$ (viewed as a manifold with boundary) is simple and denote with $\tau_U:S\bar U\rightarrow [0,\infty)$ its exit time. Note that we can write  $\chi(x,y,\lambda/x,\omega) \d \lambda \d \omega = \mu(z,v) \d v$ on $SU$ for a smooth function $\mu:SU\rightarrow \R$ with compact support. Then 
\begin{equation}\label{ldef}
L^\mu_W: C(\partial_+S\bar U)\rightarrow C(U), \quad L^\mu_W u(z) =\int_{S_zN} W^*(z,v) u^\sharp(z,v) {\mu}(z,v) \d v,
\end{equation}
where $u^\sharp$ extends $u$  constant along the geodesic flow. Further $\tilde I_W$ is the weighted $X$-ray transform, defined with respect to the manifold $\bar U$ and \eqref{factor} is evident, as $f\vert _U$ is supported in $M\cap U$ and thus no additional mass is collected by integrating along complete geodesics in $\bar U$. \\

Define $\tilde \M \subset \partial_+S\bar U$ to consist of initial conditions $(z,v)$ for which $z\in X$, the geodesic $\gamma_{z,v}(t)$ enters $B$ for some $0\le t \le \tau_U(z,v)$, but does not hit $\partial X\cap M$.   After decreasing $h$ if necessary, we can assume that 
\begin{equation}\label{sptcon3}
\{\varphi_t(z,v): 0\le t\le \tau_U(z,v)\}\cap \supp \mu = \emptyset\quad \text{ for } (z,v) \in \partial_+S\bar U \backslash \tilde \M. 
\end{equation}
Indeed, assume that $h<C_1/C_0^2\wedge \delta_0/(2C_0)$, where $\delta_0,C_0,C_1$ are the constants from \eqref{sptcon} and \eqref{sptcon2}. Then if the integral curve starting at $(z,v)\in \partial_+S\bar U$ enters the support of $\mu$ at, say $(x,y,\lambda,\omega$, we must have $0<x<2c<2h$ and $\vert \lambda/x \vert < C_0$, which implies that $\vert \lambda \vert <\delta_0$ and $x-\lambda^2/(2C_1)>x(1-x C_0^2/(2C_1))>x(1-hC_0^2/C_1)>0$. In particular the right hand side of \eqref{sptcon2} is non-negative and the curve cannot intersect $M\cap \partial X$.\newline
To proceed, take $K\subset O$ compact (such as the geodesic ball $B(p,c/2)$, when $\tilde x$ is scaled to satisfy $\vert \nabla \tilde x\vert \le 1$). We then have for all $f\in L^2(M)$
\begin{equation}\label{stabby2}
\Vert f \Vert_{L^2(K)} \lesssim_K \Vert e^{-1/x} f \Vert_{L^2(O)}  \lesssim_A \Vert x^{-\frac{d-1}{2} }e^{-1/x} L^\mu_W \tilde I_W f \Vert_{H^1(X)},
\end{equation}
where the first estimate follows from the fact that $e^{1/x}$ and all its derivatives are bounded on $K$ and the second estimate follows from 
equation \eqref{stabby} and inclusion $H_\sc^{1,\ell}(X) \subset x^{-\ell} H^{1,0}_\sc \subset x^{-\ell} H^1(X) $ for $\ell = -(d+1)/2$.\\
Note that the function on the right hand side in \eqref{stabby2}   is compactly supported in $U$ (due to the support condition on $\mu$) and that $x^{-(d-1)/2}e^{-1/x}$ and all of its derivatives extend by zero to a bounded function on $U$. Thus
\begin{equation} \label{stabby3}
\Vert f \Vert_{L^2(K)}\lesssim_{K,A} \Vert L^\mu_W \tilde I_W f\Vert_{H^1(U)},\quad \text{ for all } f \in L^2(M)
\end{equation}
and it remains to bound the operator norm of $L^\mu_W$ and relate $\tilde I_W$ to the transform $I_W$ we are actually interested in.

\begin{lemma}\label{lbounded} For all $k\ge 0$ the operator  $L^\mu_W : H^k(\tilde \M)\rightarrow H^k(U)$ is bounded with operator norm $\lesssim \Vert W \Vert_{C^k(S\bar U)}$.
\end{lemma}

\begin{proof}[Proof of Lemma  \ref{lbounded}.]
We prove the lemma in a slightly more general setting, when $\tilde \M\subset \partial_+S\bar U$ is any open subset with closure not intersecting $S\partial \bar U$ and $\mu:SU\rightarrow \R$ is a smooth function with compact support satisfying \eqref{sptcon3}.\\
The lemma then follows from factorising $L^\mu_W$ as
\begin{equation}
H^k(\tilde \M) \xrightarrow{E} H^k_c(\partial_+S\bar U ^\interior) \xrightarrow{\sharp} H^k(SU) \xrightarrow{\times \mu W^*}  H^k(SU) \xrightarrow{\pi_*} H^k(U)
\end{equation}
with the following factors: $E$ is an extension operator (cf. Lemma \ref{seeley}), which may be chosen to map to compactly supported functions, as $\tilde \M$ is assumed to have compact closure in $
\partial_+S\bar U^\interior = \partial_+S\bar U\backslash S\partial\bar U$. Due to condition \eqref{sptcon3}, the precise choice of $E$ is irrelevant.  Next, the map $\sharp$, defined below  \eqref{ldef}, is continuous, as under the isomorphism
\begin{equation}
\{(z,v,t)\in \partial_+S\bar U^\interior \times \R: 0 < t <\tau_U(z,v)\} \cong SU,\quad (z,v,t)\mapsto \varphi_t(z,v)
\end{equation} 
it corresponds to pull-back by the projection $\mathrm{pr}_1:\partial_+S\bar U^\interior \times \R \rightarrow \partial_+S\bar U^\interior$. Multiplication by $\mu W^*$ is clearly bounded with operator norm $\lesssim_\mu \Vert W \Vert_{C^k(S\bar U)}$. Finally, $\pi_*$ is the push-forward along the base-projection, which is well known (and easily checked in coordinates) to be $H^k$-continuous.
\end{proof}

\begin{lemma}\label{droptilde}
 $\Vert \tilde I_W f \Vert_{H^k(\tilde M)} \lesssim_k \Vert I_W f \Vert_{H^k(\M_O)}$ ($k\ge 0$) for all $f\in L^2(M)$.
\end{lemma}

\begin{proof}[Proof of Lemma \ref{droptilde}]
Define $\beta : \M_O\rightarrow \tilde \M$ by $\beta(z,v) = \varphi_{-\tau_U(z,-v)}(z,v)$, then $\beta^* (\tilde I_W f) = I_W f$ on $\M_O$. Now pull-back $\beta^*:H^k(\beta(\M_O)) \rightarrow H^k(\M_O)$  ($k\ge 0$) is an isomorphism, as $\beta $ extends across the closure of $\M_O$ to a diffeomorphism onto its image. Thus $\Vert \tilde I_W \Vert_{H^k(\tilde \M)} \lesssim \Vert \tilde I_W f \Vert_{H^k(\beta(\M_O))} \lesssim \Vert I_W f \Vert_{H^k(\M_O)}$, where the first inequality follows from the fact that $f$ is supported in $M$ and thus $\supp \tilde I_W f \subset \overline{\beta(\M_O)}$.
\end{proof}

We can now finish the proof of Theorem \ref{linearlocalstable}. Using \eqref{stabby3} together with the previous two lemmas yields
$
\Vert f \Vert_{L^2(K)} \lesssim_{K,A} \Vert I_Wf \Vert_{H^1(\M_O)}$ for all $f\in L^2(M)
$
and, taking $K$ to be the geodesic ball $B=B(p,c/2)$ and making the $W$-dependency explicit again,
\begin{equation}\label{stabby4}
\Vert f \Vert_{L^2(B)} \le C'(W) \Vert I_W f \Vert_{H^1(\M_O)},\quad f\in L^2(M),
\end{equation}
where $C'(W) \le \omega'(\Vert W \Vert_{C^k(SN)} \vee \Vert W^{-1} \Vert_{L^\infty(SN)}^{-1})$ for $\omega':[0,\infty)\rightarrow [0,\infty)$ non-decreasing. One can further replace the norms on $SN$ by their counterparts on $SM$, as  $\Vert I_W f \Vert_{H^1(\M_O)}$ only depends on $W\vert_{SM}$. Thus \ref{linearlocalstable1} and \ref{linearlocalstable2} are proved.\newline
Finally, part \ref{linearlocalstable3} is clear from the above: When $p$ is slightly perturbed to some $p'\in \partial M$, the ball $B(p',c/2)$ remains within $O$ and $K$ may be chosen accordingly. Small perturbations of $M$ correspond to an affine change of variables in $\tilde x$ and are thus inconsequential. This concludes the proof.
	\end{proof}

   \section{Proof of the Stability Estimate}\label{s_glob}
   
   Let $(M,g)$ be compact, non-trapping and with strictly convex boundary $\partial M$. We complete the proof of Theorem \ref{mainthm}. %We complete the proof of Theorem \ref{mainthm}, following the steps outlined in Section \ref{ss_keyideas}.
   
   \subsection{Layer Stripping Argument}  We first derive a (global) stability estimate for the linearised problem.

   %We prove the following result, which implies a global stability estimate for the weighted $X$-ray transform (if $M$ admits a strictly convex function).
   %We start by globalising the stability estimate for the weighted $X$-ray transform to arbitrary open set that satisfy the foliation condition.
	 \begin{theorem} \label{thm_glob} Let $d\ge 3$ and suppose that $K\subset O\subset M$, such that $K$ is compact and $O$ is open and satisfies the foliation condition. Then for $f\in C^\infty(M,\C^m)$ and $W\in C^\infty(SM,\Gl(m,\C))$, we have
   \begin{equation} \label{thmglob1}
   \Vert f \Vert_{L^2(K)}\le C(W)\cdot \Vert f \Vert_{C^2(M)}^{1-\mu(W)} \cdot \Vert I_Wf\Vert_{L^2(O)}^{\mu(W)},
   \end{equation}
   where $C>0$ and $\mu\in (0,1)$ obey an estimate
   \begin{equation}\label{thmglob2}
   C(W)\vee \mu(W)^{-1} \le \omega(\Vert W\Vert_{C^k(SM)}\vee \Vert W ^{-1} \Vert_{L^\infty(SM)})
   \end{equation}
   for some non-decreasing $\omega:[0,\infty)\rightarrow [0,\infty)$ and an integer $k\ge 2$.
   \end{theorem}
   
Let us outline the strategy of proof for Theorem \ref{thm_glob}. Using the strictly convex exhaustion function on $O$, we can stratify $K$ into finitely many layers, where the number of layers depends on the weight $W$. As each layer has a strictly convex boundary, one can use the local stability result in Theorem \ref{linearlocalstable} and propagate the stability estimate into the interior of $O$ layer by layer via an induction argument.
 More concretely, Theorem \ref{linearlocalstable} allows to bound the norm of $f$ within a certain layer in terms of the weighted $X$-ray transform, defined with respect to geodesics confined to that layer. As we are actually interested in the transform along complete geodesics in $M$, an error occurs. By virtue of our forward estimates, this error can be bounded in terms of the magnitude of $f$ in the previous layers, which is controlled by the induction hypothesis.
   
  \begin{remark} The H{\"o}lder-exponent $\mu$ in the theorem is of order $2^{-N}$, where $N$ is the number of layers needed to stratify $K$. This in turn  is of order $N=O(h^{-1})$, where $h$  is the `depth' from  Theorem \ref{linearlocalstable}. 
  The integer $k$ that appears in the theorem is essentially the same as in the local stability estimates (Theorem \ref{linearlocalstable}), in particular a hypothetical universal bound $k\le c_d$ in Theorem \ref{linearlocalstable} would remain true in Theorem \ref{thm_glob}.
  \end{remark}

   \begin{remark}
   	For a fixed weight $W$, the result can be improved to allow control on the H{\"o}lder-exponent $\mu$ at the cost of needing bounds on higher derivatives of $f$. Precisely, for any $\mu\in (0,1)$ we have $\Vert f \Vert_{L^2(K)} \le \omega(\Vert f \Vert_{C^\ell(M)}) \Vert I_W f \Vert_{L^2(\M_O)}^\mu$ for $\omega:(0,\infty)\rightarrow (0,\infty)$ non-decreasing (and dependent on the fixed weight $W$) and  $\ell\gg 1$ sufficiently large. To see this, one needs to amend Lemma \ref{lem_error} below by using different interpolation spaces.
   \end{remark}

	We first
	 discuss some notation and auxiliary results that are used in the proof of Theorem \ref{thm_glob}. In the following we fix a strictly convex function $\rho:O\rightarrow \R$ with 
	compact super-level sets $O_{\ge c} = \{x\in O: \rho(x) \ge c\}$ for $c>\inf_O \rho$. 	Then
	\begin{equation}\label{mc0}
	M_c =\{x\in M^\interior: \rho(x) \le c \}
	\end{equation}
	is a (possibly non-compact) manifold with strictly convex boundary and geodesics in $M_c$ with endpoints on the level set $\{\rho = c\}$ can be parametrised by the set 
\begin{equation}\label{mc}
	\M_c = \{(x,v)\in SM^
   	\interior: \rho(x)=c, \d \rho(v) \le 0, \gamma_{x,v}(\tau(x,v)) \in O\}.
%W_{c} &=&\{\gamma_{x,v}(t):(x,v)\in\M_c,0\le t\le\tau(x,v)\} \cap \{\rho \le c \}. 
%\text{ s.th. } \rho(\gamma_{x,v}(t))\ge c\},
\end{equation}
We denote with $I^c_W f :\M_c\rightarrow \C^m$ the weighted $X$-ray transform on $M_c$, defined via integrals along the portion of geodesics within $M_c$. The following Lemma compares this with the full $X$-ray transform on $M$ and provides the key estimate that drives the layer stripping argument.

\begin{lemma}[Error-bound] \label{lem_error}
	Let $f\in C^\infty(M,
	\C^m)$ and $W\in C^\infty(SM,\C^{m\times m})$, then  for all $0<\mu\le 1$ we have 
	\begin{equation} \label{lem_error1}
	 \Vert I_W^cf \Vert_{H^1(\M_c)}^2 \lesssim_\mu C(W)   \left[1+\left(\frac{\Vert f \Vert_{L^2(O_{\ge c})}}{\Vert I_W f\Vert_{L^2(\M_O)}}\right)^\mu\right] 
 \cdot \Vert f \Vert_{H^2(M)}^{2-\mu} \Vert I_Wf\Vert_{L^2(\M_O)}^\mu,
	\end{equation}
	where $C(W)>0$ is bounded when  $\Vert W \Vert_{C^2(SM)}$ is bounded.
\end{lemma}

\begin{proof} Each geodesic in $M_c$ with endpoints on the level set $\rho = c$ can be extended to a complete $O$-local geodesic in $M$ and we denote the corresponding map between initial conditions by
\begin{equation}
\beta_c : \M_c\rightarrow \M_O\subset \partial_+SM ,\quad (x,v) \mapsto \varphi_{-\tau(x,-v)}(x,v).
\end{equation}
The weighted $X$-ray transform on $M_c$ can then be written as
\begin{equation}
I^c_Wf(x,v) = I_W(1_{M_{c}} f)(\beta_c(x,v)),
\end{equation}
where $1_{M_c}$ is the indicator function of $M_c$. As $\beta$ extends smoothly to the closure of $\M_c$, pull-back by $\beta^{-1}$ defines a bounded map $H^s(\beta(\mathcal{M}_c))\rightarrow H^s(\mathcal{M}_c)$ and for all $s\in \R$ we have
 \begin{equation}
 \begin{split}
 \Vert I^c_W f\Vert_{H^s(\mathcal{M}_c)} & \lesssim_{c,s} \Vert I_W(1_{M_c}f) \Vert_{H^s(\M_O)}\\
 &\le \Vert I_W f\Vert_{H^s(\M_O)} + \Vert I_W(1_{O_{\ge c}} f)\Vert_{H^s(\M_O)}.
 %\lesssim_c \Vert I_W f \Vert_{H^s(\hat O\cap \partial_+SM)} + \Vert I_W(1_{O_{\ge c}} f) \Vert_{H^s(\M_O)}.
 \end{split}
 \end{equation}
The last term accounts for the error that is made by integrating along 
 complete geodesic in $M$ rather than the portion within $M_c$. We can bound this error by a forward-estimate (Cor.\,\ref{cor_wxray}), as long as the truncated function $1_{O_{\ge c}}f$ is of regularity $H^s$. This restricts the choice of $s$ to $\vert s \vert <1/2$, for which we obtain
 \begin{equation}
 \begin{split}
 \Vert I^c_W f\Vert_{H^s(\mathcal{M}_c)} &\lesssim_{c,s}  \Vert I_W f\Vert_{H^s(\M_O)} + \Vert W \Vert_{C^1(SM)} \Vert f\Vert_{H^s(O_{\ge c})}\\
% & \lesssim \left (\Vert W \Vert_{C^1(\hat O)} \Vert f\Vert_{H^s(O)})^{1-\mu} + \Vert W \Vert_{C^1(SM)} \frac{\Vert f\Vert_{H^s(O_{\ge c})}}{\Vert I_Wf\Vert_{L^2(\hat o \cap \partial_+SM}^\mu} \right).
 \end{split}
 \end{equation}
 In order to estimate the $H^1$-norm of  $I^c_Wf$, we employ the interpolation inequality $\Vert \cdot\Vert_{H^1}^2\le \Vert \cdot\Vert_{L^2}\Vert\cdot\Vert_{H^2}$ on $\mathcal{M}_c$ and estimate the $H^2$-term via the forward-estimate\footnote{This follows from Corollary \ref{cor_wxray}, applied to a suitable compact extension of $M_c$.}
 \begin{equation}
 \Vert I^c_W f\Vert_{H^2(\mathcal{M}_c)} \lesssim \Vert W\Vert_{C^2(SM)} \Vert f\Vert_{H^2(O)}.
 \end{equation}
\noindent Combining the estimates in the preceding displays (for $s=0$) and bounding
the first factor in  $\Vert I_Wf\Vert_{L^2(\M_O)}=\Vert I_Wf\Vert_{L^2(\M_O)}^{1-\mu}\Vert I_Wf\Vert_{L^2(\M_O)}^\mu$ via another forward estimate we get
\begin{equation}
\begin{split}
 \Vert I_W^cf \Vert_{H^1(\M_c)}^2 \lesssim&_{c,s}\left( \Vert f \Vert^{1-\mu}_{L^2(M)}\Vert W \Vert_{L^\infty(SM)}^{1-\mu} \cdot \Vert I_W f \Vert^\mu_{L^2(\M_O)} + \Vert W \Vert_{C^1(SM)} \Vert f \Vert_{L^2(O_{\ge c})}\right) \\ 
 & \times \Vert W \Vert_{C^2(SM)} \cdot \Vert f \Vert_{H^2(M)} \\
 \le & ~  (1 + \Vert W \Vert_{C^2(SM)} )^2 \cdot  \left(1 + \Vert f \Vert_{L^2(O_{\ge 0})}^\mu / \Vert I_W f \Vert_{L^2(\M_O)}^\mu\right)\\
  & \times \Vert f \Vert_{H^2(M)}^{2-\mu} \Vert I_W f\Vert_{L^2(\M_O)}^\mu,
 \end{split}
\end{equation}
as desired.\end{proof} 
 
 \begin{comment}we obtain for all $\mu \in [0,1]$:
\begin{equation}\label{glob6}
\begin{split}
\Vert I_W^cf\Vert_{H^1(\M_c)}^2 \lesssim_{A,\mu} &  ~
 \left(1+\frac{\Vert f \Vert_{L^2(O_{\ge c})}}{\Vert I_W f\Vert_{L^2(\hat O\cap \partial_+SM)}^\mu}\right) 
 \times  \Vert I_Wf\Vert_{L^2(\hat O\cap \partial_+SM)}^\mu.
\end{split}
\end{equation}
The term in brackets only depends on the $X$-ray transform on $O_{\ge c}$ and thus can be bounded once a stability estimate has been established in that region. 
\end{comment}

The next result is a technical Lemma that provides a convenient stratification of $K$  into layers. The parameter $h>0$ below will later be the `intial penetration depth' from Theorem \ref{linearlocalstable}.

   \begin{lemma}\label{lem_strat}
   Suppose $K\subset O$ and $\vert \nabla \rho \vert \ge 1$ on $K$.  
\begin{enumerate}[label=(\roman*)]
	\item \label{lem_strata}  For every $h>0$ there exists a radius $0 < r(h) \le h$ (non-decreasing in $h$) such that for $p\in K\cap \partial M_c$ with $\dist(p,\partial M)>h/2$ we have
	\begin{equation} \label{lem_strataeq}
   B(p,r(h))\cap M_c \subset \bigcup_{(x,v)\in \beta(\M_c)} \gamma_{x,v}([0,\tau(x,v)]).
   \end{equation}
   \item  \label{lem_stratb} For all $h>0$ there are finitely many numbers 
   \[
\quad \sup_K\rho =c_0 > c_1 \ge\dots> c_N > c_{N+1}=\inf_K\rho
   \quad (N =O(h^{-1}) )  
   \] as well as points $p_{ij}\in K$ ($i=0,\dots,N, j=1,\dots,J_i$) with the following properties: We have $p_{0j}\in \partial M$,  $p_{ij}\in\{\rho=c_i\}$ ($i=1,\dots, N$) and 
   \begin{equation}\label{lem_stratbeq}
   		\{x\in K : c_{i} \ge \rho(x) \ge c_{i+1}\} \subset  \bigcup_{j=1}^{J_0}B(p_{0j},h)\cup  \bigcup_{j=1}^{J_i} B(p_{ij},r) 
   		\end{equation}
   		for $i=0,\dots, N$ (where the second union is redundant for $i=0$).
\end{enumerate}   
   \end{lemma}

   \begin{proof}    Let us denote the set on the right hand side of Lemma \ref{lem_strataeq} by $V_c$. It is straightforward to see that $\partial M_c\subset V_c$ and that $V_c$ is open. In particular, $V_c$ contains an open ball around each point of $\partial M_c$. As the set of points on $\partial M_c$ with $\dist(\cdot, \partial M)\ge h/2$ is compact, the radius of the balls can be chosen uniformly (depending on $h$), which is equivalent to the statement of Lemma \ref{lem_strat}\ref{lem_strata}.\\

   \begin{figure}[h]
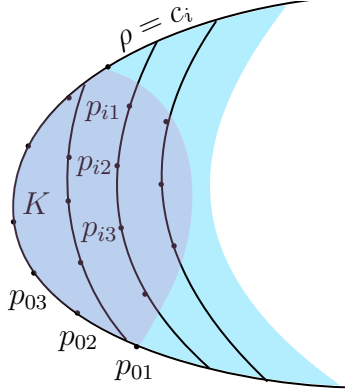   
\caption{The layers from Lemma \ref{lem_strat}}
   \end{figure}

   For part \ref{lem_stratb} we let $N(h)=2\lceil{(\sup_K \rho - \inf_K \rho)/r(h)}\rceil$ and put $c_i=c_{i-1}-r/2$ for $i=1,\dots,N$, where $c_0 = \sup_K \rho$. The boundary points $p_{01},
  \dots, p_{0J_0}$ are then chosen such that the $h$-balls around them cover the compact set $\partial M\cap K$. Now let $x\in K$ be such that $c_i \le \rho(x) \le c_{i+1}$ for some $i=0,\dots, N$. If $i=0$ or $\dist(x,
  \partial M) < h/2$, then $x \in B(p_{0,j},h)$ for some $j=1,\dots, J_0$. If $i\ge 1$ and $\dist(x,
  \partial M) \ge h/2$ %, in order to establish \eqref{lem_stratbeq}, 
  we claim that $d(x,p)<r$ for some $p\in \partial{M_{c_i}}$. Due to the compactness of $\partial M_{c_i} \cap \{\dist(\cdot ,\partial M) \ge h/2\}$, finitely many such points $p_{i1},\dots, p_{iJ_i}\in \partial M_{c_i}$ suffice to establish \eqref{lem_stratbeq}, so it remains to verify the claim.\\
  Indeed, if we let $t\mapsto c_x(t)$ be the unit-length curve with $c_x(0)=x$ and $\d \rho(\dot c_x(t)) = \vert \nabla \rho(c_x(t)) \vert$, then $\rho$ increases along $c_x$ and by \cite[Lem,\,2.5]{PSUZ19} the curve stays in $O$ until it hits the boundary of $M$. Let $\ell \ge 0$ be the first time for which $p=c_x(\ell) \in \partial M \cup \partial M_{c_i}$. Then 
  \begin{equation}
  d(x,p) \le \ell \le \int_0^\ell \d \rho(\dot c_x(t) ) \d t = \rho(p) - \rho(x) \le c_i - c_{i+1} \le r/2
  \end{equation}
  and we must have $p \in \partial M_{c_i}$ and $x\in B(p,r)$, as desired.
   \end{proof}
   
	The next Lemma is of importance for the full data problem ($O=M$) and allows to perturb  convex foliations in a way that shifts the point of degeneracy.

   \begin{lemma}\label{lem_perturb} Suppose $\rho:M\rightarrow \R$ is smooth and strictly convex. Then there exists another $\tilde \rho:M\rightarrow \R$, smooth and strictly convex, such that $\rho$ and $\tilde \rho$ achieve their global minima at different points. 
   \end{lemma}
   \begin{proof}
Suppose $\rho$ achieves its minimum at the point $x^*\in M$ and let $V$ be a smooth vector field on $M$ which is tangent to $\partial M$ and non-vanishing at $x^*$. Denote the flow of $V$ by $(\psi_t:t\ge 0)$, then $\psi_t^*\rho\in C^\infty(M,\R)$ is strictly convex for $t$ sufficiently small and achieves its (unique) minimum at $x^*_t=\psi_{-t}(x^*)$. Since $V(x_*)\neq 0$, we have $x_t^* \neq x^*$ for $t>0$ sufficiently small and thus $\tilde\rho = \psi_t^*\rho$ and $\rho$ achieve their minimum at different points.
   \end{proof}

	\begin{proof}[Proof of Proposition \ref{thm_glob}]
	Let $\rho:O\rightarrow \R$ be a strictly convex exhausting function and denote  $\rho_*=\inf_O\rho$. We first reduce to the situation that 
	\begin{equation}\label{glob1}
	\vert \nabla \rho(x) \vert \ge 1 \quad \text{ and } \quad \rho(x)>\rho_*\quad \text{for } x\in  K.
	\end{equation} 
	Indeed, after scaling $\rho$ if necessary, \eqref{glob1} can only fail, when $O=M$ \cite[Lemma 2.5]{PSUZ19} and in this case we argue as follows:  Take $\tilde \rho$ as in Lemma \ref{lem_perturb},  then we may choose $\epsilon>0$ such that $M=\{\rho \ge \rho_*+\epsilon\}\cup \{\tilde\rho \ge \inf_M\tilde \rho+\epsilon\}=K\cup \tilde K$. Then $K$ and $\tilde K$ satisfy \eqref{glob1} for $\rho$ and $\tilde \rho$ respectively and the corresponding stability estimates \eqref{thmglob1} can be combined to bound $
	\Vert I f\Vert_{L^2(M)}$. \newline
	In the remaining proof we argue with fixed  $f\in C^\infty(M,\C^m)$ and $W\in C^\infty(SM,\Gl(m,\C))$ and keep track of the dependency of our constructions on 
	\begin{equation}
	A = \Vert W\Vert_{C^k(SM)} \vee \Vert W^{-1}  \Vert_{L^\infty(SM)} %\le A
	\end{equation}
	for an integer $k\ge 2$ to be specified. Let us first summarise the consequences of Theorem \ref{linearlocalstable}:
	 %We will write `$\lesssim_A$' for an inequality with implicit constant depending on $A$, but otherwise independent of the involved functions.\\
		Each $p\in K$ is a strictly convex boundary point of either $M$ itself or of the manifold $M_c$, defined in \eqref{mc0}. We can thus apply Theorem \ref{linearlocalstable}, either with respect to the local $X$-ray transform $I_W$ on $M$ or  the one on $M_c$, which we denote by $I^c_W$. Thus, for all $f\in L^2(M,\C^m)$ we have
	\begin{eqnarray}
	\Vert f \Vert_{L^2(B(p,h))} \le & C \Vert I_Wf\Vert_{H^1(\M_O)}, &  p\in K\cap \partial M \label{glob2} \\
	\Vert f \Vert_{L^2(B(p,r)\cap M_{c} })  \le &  C \Vert I^c_W f \Vert_{H^1(\M_c)}, &  p\in K\cap \partial M_c\backslash B(\partial M,h/2)
	\label{glob3}
	\end{eqnarray}
	where $C,h>0$ depend on $W$ and $r=r(h)$ is as in Lemma \ref{lem_strat}\ref{lem_strata}. By part \ref{linearlocalstable3} of the theorem, the choice of regularity $k$ that appears in \eqref{linearlocalstable2eqn} can be made uniform over the compactum $K$, and will be fixed from now on (assuming $k\ge 2$ without loss of generality). We then have
$
	C\vee h^{-1} \lesssim_A 1.
	$\newline
We proceed by stratifying $K$ into layers $\{x\in O:c_i\ge \rho(x) > c_{i+1}\}$ ($i=0,\dots, N$) for $c_0,\dots,c_{N+1}$  as in Lemma \ref{lem_strat}\ref{lem_stratb} with $N\lesssim h^{-1} \lesssim_A 1$. % (note that due to \eqref{glob1}, the condition $\vert \nabla \rho \vert \ge 1$ on $K$ can be achieved by scaling).  
We will prove inductively that
\begin{equation}\label{glob7}
\Vert f \Vert_{L^2(O_{c_i})} \lesssim_A \Vert f \Vert_{C^2(M)}^{1-2^{-i}}\Vert I_W f\Vert_{L^2(\M_O)}^{2^{-i}},\quad i=1,\dots, N+1
\end{equation}
which implies \eqref{thmglob1}. For $i=0$ this is a straightforward consequence of \eqref{glob2}. Indeed, for every $p\in \partial M$, we can use the interpolation inequality $\Vert \cdot\Vert_{H^1}^2\le \Vert \cdot\Vert_{L^2}\Vert\cdot\Vert_{H^2}$ on $\M_O$ and a forward estimate (Thm.\,\ref{cor_wxray}) to obtain
\begin{equation}\label{glob8}
\Vert f \Vert_{L^2(B(p_{0j},h))}    \lesssim_A \Vert f \Vert_{C^2(M)}^{1/2} \Vert I_Wf\Vert^{1/2}_{L^2(\M_O)},\quad j=1,\dots, J_0,
\end{equation}
where the points $p_{01},\dots, p_{0J_0}$ are as in Lemma \ref{lem_strat}\ref{lem_stratb}. As the corresponding $h$-balls cover $O_{\ge c_1}$, this implies \eqref{glob7} for $i=1$.\\
Next assume the estimate has been established for some $1\le i < N$ and consider the points $p_{i1},\dots,p_{iJ_i}$ from the Lemma. Then \eqref{glob3} and Lemma \ref{lem_error}, combined with the induction hypothesis which allows to bound the bracketed term in  \eqref{lem_error1}, yield
\begin{equation}
\begin{split}
\Vert f\Vert_{L^2(B(p_{ij},r)\cap M_{c_i})}  \lesssim_A \Vert f \Vert_{C^2(M)}^{1-2^{-(i+1)}} \Vert I_W f\Vert_{L^2( \M_O)}^{2^{-(i+1)}}.
\end{split}
\end{equation}
A similar bound can be achieved on the balls $B(p_{0j},h)$ (decreasing the H{\"o}lder-exponent as in the proof of Lemma \ref{lem_error}) and together with the induction hypothesis we conclude \eqref{glob7} for $i+1$. This finishes the proof.
\end{proof}

\begin{comment} \eqref{glob8} and after adjusting the exponent to $2^{-(i+1)}$ as explained above \eqref{glob6}, we obtain a similar bound as in the previous display.  Since finitely many $p\in \partial M\cap \partial M_{c_i}$ suffice to bound $K\cap \{c_{i+1} < \rho \le c_{i}\}$, this gives \eqref{glob7} for $i+1$.\\
Evaluating \eqref{glob8} for $i=N$ concludes the proof. For matrix weights $W$ satisfying $\inf_{p\in K} h(p,W)>h$ the H{\"o}lder-exponent is then given by $\mu = 2^{-\lceil h^{-1} \rceil}$.
\end{comment}

   \subsection{Proof of Theorem \ref{mainthm}} We conclude the main stability theorem by combining the linear estimates from the previous section with pseudo-linearisation formula and the bounds on integrating factors from Theorem \ref{thm_mainintfac}.
   
   \begin{proof}[Proof of  Theorem \ref{mainthm}]
   Let $\Phi,\Psi \in C^
\infty(M,\C^{m\times m}) $ and recall from Lemma \ref{lem_pseudolin}, that $C_\Phi - C_\Psi =R_\Phi I_{\mathcal{W}_{\Phi,\Psi} }(\Phi - \Psi) \alpha^* R_\Psi^{-1}$, where $\mathcal{W}_{\Phi,\Psi}A = R_{\Phi}^{-1} A R_\Psi$ and we may choose smooth integrating factors $R_\Phi$ and $R_\Psi$ as in Theorem \ref{thm_mainintfac}.\\
Now for $K\subset O \subset M$ as in the theorem, we can apply Theorem \ref{thm_glob} to obtain
\begin{equation}
\Vert \Phi - \Psi \Vert_{L^2(K)} \le C(\mathcal{W}_{\Phi,\Psi}) \cdot \Vert \Phi-\Psi \Vert_{C^2(M)}^{1-\mu(W)}\cdot \Vert C_\Phi - C_\Psi \Vert_{L^2(\M_O)}^{\mu(\mathcal{W}_{\Phi,\Psi})}
\end{equation}
with $C(\mathcal{W}_{\Phi,\Psi})  \vee \mu(\mathcal{W}_{\Phi,\Psi})^{-1}$ bounded above by
\begin{equation}
 \omega(\Vert \mathcal{W}_{\Phi,\Psi} \Vert_{C^k(SM)} \vee \Vert \mathcal{W}_{\Phi,\Psi}^{-1} \Vert_{L^\infty(SM)})
\end{equation}
for a non-decreasing function $\omega:[0,\infty)\rightarrow [0,\infty)$. It remains to bound the norms in the previous display in terms of $\Vert \Phi \Vert_{C^k(M)} \vee \Vert \Psi \Vert_{C^k(M)}$. Note that $\Vert \mathcal{W}^{\pm 1}_{\Phi,\Psi} \Vert_{C^k(SM)} \lesssim \Vert R_\Phi^{\mp 1} \Vert_{C^k(M)} \cdot \Vert R_\Psi^{\pm 1}\Vert_{C^k(M)}$, hence the proof is finished by using the bounds from Theorem \ref{thm_mainintfac}.
 \end{proof}
   
   \section{Statistical application} \label{s_stats}
	In this section we demonstrate the scope of our stability estimate (Theorem \ref{mainthm}) by showing how it can be used to establish a statistical consistency result. We will focus on the full data problem ($O=M$) and discuss the two dimensional results from \cite{MNP19} alongside with 
	the case $d\ge 3$.
	Let us therefore assume
	 $(M,g)$ is a compact Riemannian manifold with strictly convex boundary and that we are in either of the following cases:
\begin{enumerate}[label=(\Alph*)]
\item \label{caseA}$d=2$ and $M$ is simple%\footnote{Simple means that $M$ is non-trapping and free of conjugate points. In dimension $2$ this is implies the existence of a strictly convex function.}
\item \label{caseB} $d\ge 3$ and $M$ admits a strictly convex function
\end{enumerate}
In both cases we assume for simplicity\footnote{For $d=2$ and $d=3$ no generality is lost, as any manifold satisfying \ref{caseA} or \ref{caseB} is automatically diffeomorphic to a Euclidean ball. In higher dimensions this might fail, but the author is not aware of any counterexamples.
} that (as a smooth manifold) $M$ is the closed unit-ball in $\R^d$. We further assume that the potentials $\Phi$ take values in either $\so(m)=\{A\in \R^{m\times m}: A^T = - A\}$ or $\gl_m(\R)=\R^{m\times m}$ and write $\g$ to denote either choice.
\\
The plan for this section is as follows: We first record all necessary estimates at one place, then give a brief overview of the Bayesian approach of inverse problems and recall the main statistical theorem from \cite{MNP19}, including a sketch of its proof. Finally, in the last subsection, we explain how the proof can be amended to obtain a consistency result in case \ref{caseB}. \\
In order to keep the overlap with \cite{MNP19} at a minimum, the discussion below is brief and heavily relies on \cite{MNP19}. For more background on the statistical framework we refer to the books \cite{GhVa17} and \cite{GiNi16}.

\subsection{Available Estimates} %We collect here all relevant estimates for the non-abelian $X$-ray transform, contrasting the situation between cases \ref{caseA} and \ref{caseB}.\\
In both cases the available forward- and stability-estimates  take the following form:  %valid for smooth potentials $\Phi,\Psi : M\rightarrow \g$:
\begin{eqnarray}
\Vert C_\Phi - C_\Psi \Vert_{L^2(M)} &\le& c_1(\Phi,\Psi) \cdot \Vert \Phi - \Psi  \Vert_{L^2(M)}  \label{est1} \\
\Vert C_\Phi \Vert_{L^\infty(M)}&\le& c_2(\Phi) \label{est2} \\
\Vert \Phi - \Psi \Vert_{L^2(M)} & \le& C(\Phi,\Psi) \cdot \Vert C_\Phi - C_\Psi \Vert_{L^2(\partial_+SM)}^{\mu(\Phi,\Psi)}, \label{est3}
\end{eqnarray}
Here $c_1(\Phi,\Psi),c_2(\Phi),C(\Phi,\Psi)>0$ and $\mu(\Phi,\Psi)\in (0,1)$ may depend on the potentials. The validity of the estimates and the uniformity properties of the constants can be summarised as follows:  
\begin{itemize}
\item  The forward estimates \eqref{est1} and \eqref{est2} are the same in case \ref{caseA} and \ref{caseB} and hold true for smooth potentials $\Phi,\Psi:M\rightarrow \g$. If $\g = \so(m)$, then $c_1$ and $c_2$ are constant, due to the compactness of $SO(m)$. If $\g = \gl_m(\R)$, then $c_1(\Phi,\Psi)$ and $c_2(\Phi)$ are uniform on $L^\infty$-balls.
\item In case \ref{caseA} and for $\g = \so(m)$ we can choose any integer $k\ge 2$. Then \eqref{est3} holds true for smooth $\Phi,\Psi:M\rightarrow \so(m)$ with $\mu(\Phi,\Psi) = (k-1)/k$ and  $C(\Phi,\Psi)$ uniform on $C^k$-balls. 
\item In case \ref{caseB}  and for $\g = \gl_m(\R)$ there exists an integer $k\gg 1$ such that \eqref{est3} holds true for $\Phi,\Psi:M\rightarrow \gl_m(\R)$ with both $C(\Phi,\Psi)$ and $\mu(\Phi,\Psi)$ uniform on $C^{k}$-balls. 
\end{itemize}
Here we say that a quantity is `uniform on $F$-balls' (for $F=L^\infty(M,\g)$ or $F=C^k(M,\g)$) if its supremum (resp. infimum) over $\{\Phi,\Psi: M\rightarrow \g \text{ smooth} : \Vert \Phi \Vert_F \vee \Vert \Psi\Vert_F \le A\}$ is finite (resp. $>0$) for all $A>0$.\\
The forward estimates are proved in Corollary \ref{cor_naxray} for a general non-trapping manifold (with strictly convex boundary). The stability estimate for case \ref{caseB} is the content of our main theorem (Thm.\,\ref{mainthm}) and the version for case \ref{caseA} is discussed below the main theorem.

\begin{remark}\label{unknown1} An important difference between case \ref{caseA} and \ref{caseB} lies in the role of `regularity parameter' $k$ and H{\"o}lder-exponent $\mu$, which -in the statistical analysis below- determine the choice of prior and the rate of contraction respectively. In case \ref{caseA} one can effectively choose the H{\"o}lder exponent arbitrarily close to $1$ (by sending $k\rightarrow \infty$), %and at the cost of having to bound higher and higher derivatives of the potentials) 
while in case \ref{caseB}, our method of proof yields an unknown  $k$ and  there is no control over the H{\"o}lder-exponent. See also remark \ref{unknown2}.

\end{remark}

\subsection{Statistical Background} 
The statistical question we are concerned with arises in following experimental setup:
Suppose for $\Phi\in C(M,\g)$ we observe  the data $(X_i,V_i,Y_i)_{i=1}^n$, where 
\begin{equation}\label{experiment}
	Y_i = C_\Phi(X_i,V_i) + \epsilon_i,\quad i=1,\dots,n,
\end{equation}
with directions $(X_i,V_i)$ ($i=1,\dots,n$) drawn independently and uniformly\footnote{Uniform here means that the law of $(X_i,Y_i)$ is the standard Riemannian volume-form on $\partial_+SM$, normalised to have mass $1$.} from $
\partial_+SM$  and independent additive noise given by
\begin{equation}
\epsilon_i= (\epsilon_{ij}:1\le j\le \dim \g)\in \R^{\dim \g}\equiv \g \quad \text{ for } \epsilon_{ij}{\sim} N(0,1) \text{ i.i.d}.
\end{equation}
We write $P^n_\Phi=\L(D_n\vert \Phi)$ for the law of $D_n=(X_i,V_i,Y_i:1\le i\le n)$, arising from \eqref{experiment} with potential $\Phi$. The statistical experiment just described is then encoded in the collection of probability measures $(P_\Phi^n:\Phi \in C(M,\g))$ on the sample space $\D^n=(\partial_+SM\times \g)^n$. \\
The Bayesian approach to estimate $\Phi$ from a sample $D^n=((X_i,V_i,Y_i): 1\le i \le n)\in \D^n$ is to choose a prior $\Pi_n$ on $C(M,\g)$ and compute the posterior probability under the sample $D^n$ of a (Borel-measurable) set $B\subset C(M,\g)$ according to the formula
\begin{equation} \label{posterior}
\Pi_n(\Phi \in B\vert D^n) = \frac{\int_B p^{n}_\Phi(D^n) \Pi(\d \Phi) }{\int p^n_\Phi(D^n) \Pi(\d \Phi)},
\end{equation}
where $p_\Phi^n(D^n)$ is the likelihood of $D^n$ being generated from $\Phi$. Precisely, $p_\Phi^n=p_\Phi^1 \otimes \dots \otimes p_\Phi^1$ ($n$-times), where
$
\log p^1_\Phi(x,v,y) =  -\frac12\vert C_\Phi(x,v) - y \vert_F^2 -\frac{\dim \g}{2} \log(2\pi)
$
(for $(x,v,y)\in \D^1$) and $\vert \cdot \vert_F$ is the Frobenius-norm.\\
Given the posterior one can estimate $\Phi$, for example by the posterior mean which in our setting exists as Bochner-integral in $C(M,\g)$. From a frequentist perspective one then asks how well $\Phi$ is estimated, when the data is generated from \eqref{experiment} with a `true' potential $\Phi_0$ and a first such quality assessment is given by the posterior consistency results below.

\subsection{Posterior consistency in case \ref{caseA}} In order to state the posterior consistency result of \cite{MNP19}, we first review the construction of priors (in arbitrary dimension $d\ge 2$), focusing on their key example  based on Mat{\'e}rn-Whittle-processes.\\
For a given choice of regularity parameter $\alpha > d/2$, define a base prior $\underline{\Pi}=\underline{\Pi}(\alpha)$ on $C(M,\R)$  as law of a centred Gaussian process $(f(x):x\in M)$ with covariance
$
\mathbb{E}f(x) f(y) = \int_{\R^d} e^{i(x-y)\xi}  \langle \xi \rangle^{-2\alpha}\d\xi,
$
where it is understood that $M\subset \R^d$. This so called {\it Mat{\'e}rn-Whittle-process of regularity $\alpha$} is a standard prior choice in %the context of 
non-parametric Bayesian statistics (Example 11.8 in \cite{GhVa17}) and satisfies
\begin{equation}\label{rkhs}
\mathrm{RKHS}(\underline{\Pi}) = H^\alpha(M,\R),\quad \underline{\Pi}(C^k(M,\R)) = 1\text{ for } k\in \Z\cap [0,\alpha-d/2)%0\le k< \alpha-d/2,
\end{equation}
where $\mathrm{RKHS}(\cdot)$ stands for the `reproducing kernel Hilbert-space'.\\% of a Gaussian measure.\\
The prior $\Pi_1$ on $C(M,\g)$ is then obtained by drawing each component (in an identification $\g\equiv \R^{\dim \g}$) independently from $\underline{\Pi}$. For $n\ge 2$ the prior $\Pi_n$ is defined by scaling $\Pi_1$, precisely
\begin{equation}\label{prior}
\Pi_n = \L\left( n^{-\frac{d}{4\alpha+ 2d}} \Phi \right),\quad \text{ for } \Phi \sim \Pi_1,
\end{equation}
where $\L(\cdot )$ denotes the law of a random variable. \\

Then in case \ref{caseA} ($M$ is a simple surface), the following result holds true:

\begin{theorem}[Thm.\,3.2 in \cite{MNP19}] \label{thmposcon} Suppose we are in case \ref{caseA} above, $\g = \so(m)$ and $\alpha >3$. Then for every $\Phi_0 \in C^\infty(M,\so(m))$, there is a $\gamma>0$ such that
\begin{equation}\label{posteriorcontraction}
\Pi_n(\Phi: \Vert \Phi - \Phi_0 \Vert_{L^2(M)} \ge n^{-\gamma}\vert D^n   ) \rightarrow 0\quad  \text{ as } n\rightarrow \infty
\end{equation}
 in $P^n_{\Phi_0}$-probability. Here $\Pi_n(\cdot\vert D^n)$ are the posteriors, defined in \eqref{posterior}, with respect to the 
scaled Mat{\'e}rn-Whittle-priors in \eqref{prior}  of regularity $\alpha$.
\end{theorem}

\begin{remark}[Generalisations]
The theorem remains true for a larger class of base-priors (specified in \cite[Condition 3.1]{MNP19}). Further, the regularity of $\Phi_0$ can be relaxed and, by varying $\alpha$, 
one has control over the rate of contraction $\gamma$ (Remark 3.3 in \cite{MNP19}).
\end{remark}

\begin{remark} The scaling rate $\nu={d}/(4\alpha+ 2d)$ in\eqref{prior} is chosen such that, writing $t_* = 2t/(2+t)$ for $t>0$, we have
\begin{equation}\label{entropyrate}
(4\nu / (1-4\nu ) )_* = d/\alpha,
%\left( \frac{4\nu}{1-4\nu}\right)_*=\frac{d}{\alpha}, 
% \equiv ~ \left \{ %\begin{matrix}
%\text{$L^2$-entropy exponent of the unit-ball} \\
%\mathbb{B}^\alpha\subset H^\alpha(M,\g)= \mathrm{RKHS} (\Pi_1) .
%\end{matrix}\right\}\footnote{
%See Lemma \ref{metricentropy} in the appendix.}
\end{equation} 
which arises as exponent in a classical $L^2$-entropy bound for the unit-ball $\mathbb{B}^\alpha \subset H^1(M,\g)= \mathrm{RKHS} (\Pi_1)$ (cf.\,Lemma \ref{metricentropy}).
\end{remark}

\begin{proof}[Sketch of proof]  
Let $\delta_n = n^{-\alpha/(2\alpha + d)}$($=n^{\nu-1/2}$). Using \eqref{entropyrate} and a  theorem of Li-Linde \cite[Thm.\,1.2]{LiLi99}, one computes the small ball probability
\begin{equation}\label{smallball}
% - \log \Pi_n(\Vert \Phi - \Phi_0 \Vert_{L^2(M)} \le \delta_n ) \lesssim n\delta_n^2
  -  \log \Pi_n(\Vert \Phi\Vert_{L^2(M)} \le \delta_n ) \lesssim n\delta_n^2.
\end{equation} 
The event in the last probability can be changed to $\Vert \Phi - \Phi_0\Vert_{L^2(M)} \le \delta_n$ by a standard argument (Anderson's Lemma, cf. \cite[Cor.\,2.6.18]{GiNi16}) and expressed in terms of the likelihoods $p^n_\Phi,p^n_{\Phi_0}$ by using the forward estimates. A general contraction theorem  (\cite[Thm.\,5.13]{MNP19}) then implies that, for some sufficiently large $m'>0$, we have 
\begin{equation} \label{conhellinger}
\Pi_n(\Phi: h(p^n_\Phi,p^n_{\Phi_0})\le m' \delta_n \vert D^n) \xrightarrow{{P^n_{\Phi_0}}} 1,\quad \text{as } n\rightarrow \infty.
\end{equation}
Here $h(p^n_\Phi,p^n_{\Phi_0})$ denotes the Hellinger-distance, which is 
$\approx \Vert C_\Phi - C_{\Phi_0} \Vert_{L^2(\partial_+SM)}$, as the scattering data is $SO(m)$-valued (\cite[Lem.\,5.14]{MNP19}).\\ 
By \eqref{rkhs} it follows for $0\le k < \alpha -d/2$ that the events $\F'(A)=\{\Vert \Phi \Vert_{C^k(M)} \le A \}$ ($A>0$) have $\Pi_n$-mass approaching $1$ as $n\rightarrow \infty$ (Fernique's theorem, cf. \cite[Thm.\,2.1.20]{GiNi16}), which suggests that one can intersect the event in \eqref{conhellinger} with $\F'(A)$ without destroying the limit. To make this precise one shows, using Borell's isoperimetric inequality \cite{Bor75}, that the slightly smaller events
$
\F_n(A) = \{\Phi_1 + \Phi_2: \Vert \Phi_1\Vert_{L^2(M)} \le \delta_n, \Vert \Phi_2 \Vert_{H^\alpha(M)}\le A\}\cap \F'(A)
$
obey \begin{equation} \label{smallf}
-\log \Pi_n(\F_n(A)^c) \ge \omega(A) n\delta_n^2\quad \text{ and }\quad  \log \mathcal{N}(\F_n(A),h,
\delta_n)\lesssim_A n\delta_n^2
\end{equation}
with $\omega(A)$ unbounded and non-decreasing in $A$  (\cite[Lem.\,5.17]{MNP19}) and where $\log \mathcal{N}$ is the metric entropy, defined above Lemma \ref{metricentropy}.
Then, for $A>0$ sufficiently large, \cite[Thm.\,5.13]{MNP19} indeed implies that
\begin{equation}\label{posconc}
\Pi_n(\Phi: \Vert C_\Phi - C_{\Phi_0} \Vert_{L^2(\partial_+SM)}  \le A \delta_n, \Vert \Phi \Vert_{C^k(M)} \le A\vert D_n)\xrightarrow{{P^n_{\Phi_0}}} 1,
\end{equation}
as $n\rightarrow \infty$ \cite[Thm.\,5.19]{MNP19}. 
 If $\alpha > 3$,  we may choose  $k \in \Z \cap [2,\alpha - d/2)$
and apply stability estimate \eqref{est3} with H{\"o}lder-exponent $(k-1)/k$. Thus on the event in the previous display we have
\begin{equation}
\Vert \Phi - \Phi_0 \Vert_{L^2(M)} \le (A' \delta_n)^{(k-1)/k}
\end{equation}
for some $A'>0$ which incorporates the constant from the stability estimate.  Choosing a slightly slower rate $0<\eta<(k-1)/k$, the constant $A'$ can be absorbed in the limit $n\rightarrow \infty$ and thus 
\begin{equation}
\Pi_n(\Phi:\Vert \Phi - \Phi_0 \Vert_{L^2(M)} \le \delta_n^{\eta}, \Vert \Phi \Vert_{C^k(M)} \le A' \vert D_n) \rightarrow 1 
\end{equation}
in $P_{\Phi_0}^n$-probability. Dropping the constraint $\Vert \Phi \Vert_{C^k(M)} \le A'$ yields \eqref{posteriorcontraction} and finishes the proof.
\end{proof}

 \begin{comment}
 
 for $0<k<\alpha-1$ the events $\{\Phi: \Vert \Phi \Vert_{C^k(M)} \le m''\}$ ($m''>0$ large) have $\Pi_n$-mass rapidly approaching $1$ as $n\rightarrow \infty$ (Fernique's theorem, cf. \cite[Thm.\,2.1.20]{GiNi16}) and can thus be introduced in \eqref{conhellinger} without changing the limit (\cite[Thm.\,5.10]{MNP19}). Summarising, if $m''>0$ is large enough, we have
\begin{equation} \label{posconc}
\Pi_n(\Phi: \Vert C_\Phi - C_{\Phi_0} \Vert_{L^2(\partial_+SM)}  \le m'' \delta_n, \Vert \Phi \Vert_{C^k(M)} \le m'' \vert D_n)\xrightarrow{{P^n_{\Phi_0}}} 1,
\end{equation}
%which is the content of 
as $n\rightarrow \infty$ \cite[Thm.\,5.16]{MNP19}. 

 If $\alpha > 3$,  one may choose an integer $k\in [2,\alpha-1]$
and apply stability estimate \eqref{est3} with H{\"o}lder-exponent $(k-1)/k$. Thus on the event in the previous display we have
\begin{equation}
\Vert \Phi - \Phi_0 \Vert_{L^2(M)} \le (m''' \delta_n)^{(k-1)/k}
\end{equation}
for some $m'''>0$ which incorporates the constant from the stability estimate.  Choosing a slightly slower rate $0<\gamma<(k-1)/k$, the constant $m'''$ can be absorbed in the limit $n\rightarrow \infty$ and thus 
\begin{equation}
\Pi_n(\Phi:\Vert \Phi - \Phi_0 \Vert_{L^2(M)} \le \delta_n^{\gamma}, \Vert \Phi \Vert_{C^k(M)} \le m' \vert D_n) \rightarrow 1 
\end{equation}
in $P_{\Phi_0}^n$-probability. Dropping the constraint $\Vert \Phi \Vert_{C^k(M)} \le m'$ yields \eqref{posteriorcontraction} and finishes the proof.
\end{proof}
\end{comment}
\subsection{Posterior consistency in  case \ref{caseB} }  The proof above can be adapted to case \ref{caseB} ($M$ of dimension $d \ge 3$, supporting a strictly convex function) and $\g = \gl_m(\R)$ to obtain the following result:

\begin{theorem}\label{Bcont}
Suppose we are in case \ref{caseB} above and $\g=\gl(m)$. Then there exist $\alpha>0$ and $\gamma>0$, such that for all $\Phi_0\in C^\infty(M,\g)$ we have
\begin{equation}\label{Bcont1}
\Pi_n(\Phi: \Vert \Phi - \Phi_0 \Vert_{L^2(M)} \ge n^{-\gamma} \vert D_n) \rightarrow 0\quad \text{ as } n\rightarrow \infty
\end{equation}
in $P^n_{\Phi_0}$-probability. Here $\Pi_n(\cdot\vert D_n)$ is again the posterior defined in \eqref{posterior} with respect to the scaled Mat{\'e}rn-Whittle-priors in \eqref{prior} of regularity $\alpha$.\hfill \qed
\end{theorem}

Under the hypotheses of the theorem and essentially with the same arguments as in \cite{MNP19} one can use the theorem above to derive a consistency result for the posterior mean. This is defined as $\bar \Phi_n (D_n) = E_{\Pi_n}[\Phi\vert D_n]$ and exists as Bochner-integral in $C(M,\g)$. Using the precise exponential convergence rate in \eqref{Bcont1} (above withhold for simplicity), one then shows that \begin{equation}
P_{\Phi_0}\left( \Vert \bar \Phi_n(D_n) - \Phi_0 \Vert_{L^2(M)} > n^{-\gamma}\right)\rightarrow 0,\quad \text{ as } n\rightarrow \infty,
\end{equation}
which gives precisely Theorem \ref{consistencymain} as stated in the introduction.

\begin{remark}\label{unknown2}
In comparison with Theorem \ref{thmposcon}, the theorem has two shortcomings: First, the rate of contraction, while being polynomial, is unknown. Second, and more importantly, the required regularity of the prior (the choice of $\alpha$) is unknown as well and thus the theorem does {\it not} provide a precise guideline for the choice of prior in applications.\\
Possibly the latter issue can be alleviated by choosing a prior with $C^\infty$-smooth sample paths, such as a squared exponential prior. However, as  our ignorance of $\alpha$ rather seems to be an artefact of the proof of the underlying stability estimate than an intrinsic feature of the inverse problem, it is questionable whether such a prior choice is advisable.
\end{remark}

\begin{proof}[Sketch of proof of Theorem \ref{Bcont}] Let us first discuss the case $\g = \so(m)$. Then, as we have identical forward estimates as in case \ref{caseA} and the general contraction theory is independent of the dimension, the proof of Theorem \ref{thmposcon} extends verbatim to case \ref{caseB} up to equation \eqref{posconc}. That is, for $A>0$ large enough (and $0\le k < \alpha -d/2)$ we have, as $n\rightarrow \infty$
\begin{equation} \label{posconcrep}
\Pi_n(\Phi: \Vert C_\Phi - C_{\Phi_0} \Vert_{L^2(\partial_+SM)}  \le A \delta_n, \Vert \Phi \Vert_{C^k(M)} \le A\vert D_n)\xrightarrow{{P^n_{\Phi_0}}} 1.
\end{equation}
To proceed, one chooses $\alpha>0$ so large, that $\alpha - d/2$ exceeds the regularity parameter $k$ from Theorem \ref{mainthm}. Then stability estimate \eqref{est3} implies that on the event in \eqref{posconcrep} we have
\begin{equation}
\Vert \Phi - \Phi_0 \Vert_{L^2(M)} \le (A'\delta_n)^{\mu},
\end{equation}
where $A'$ incorporates the constant from the stability estimate and (in the notation of \eqref{est3}) $\mu = \inf \mu(\Phi,\Phi_0)>0$, where the infimum is taken over $\{\Phi: \Vert \Phi \Vert_{C^k(M)} \le A\}$. The proof is then finished as in case \ref{caseA}.\newline
For $\g = \gl_m(\R)$, \eqref{posconcrep} remains true, but one has to take some care in its derivation, as the scattering data now assumes values in the non-compact group $\Gl(m,\C)$ and the forward estimates are only uniform on $L^\infty$-balls. We will explain the necessary changes in the following:\\
As for the small ball probabilities, \eqref{smallball} has to be replaced by
\begin{equation}
-\log \Pi_n(\Vert \Phi \Vert_{L^2(M)} \le \delta_n, \Vert \Phi\Vert_{L^\infty} \le A ) \lesssim_A n\delta_n^2,
\end{equation}
which follows from \eqref{smallball} and the Gaussian correlation inequality \cite{LaMa17}
\begin{equation*}
\begin{split}
\Pi_n(\Vert \Phi \Vert_{L^2(M)} \le \delta_n, \Vert \Phi\Vert_{L^\infty} \le A ) & \ge \Pi_n(\Vert \Phi \Vert_{L^2(M)} \le \delta_n) \Pi_n( \Vert \Phi\Vert_{L^\infty(M)} \le A ),
\end{split}
\end{equation*}
noting that $-\log \Pi_n(\Vert \Phi \Vert_{L^\infty(M)} \le A) = o(1)$ as $n\rightarrow \infty$ due to Fernique's theorem. Mutatis mutandis, the same arguments as in case \ref{caseA}  imply \eqref{conhellinger}.\\
Next, the comparison between Hellinger- and $L^2$-distance in the general case (and with essentially the same proof) takes the form
\begin{equation}\label{hellingercomparison}
\omega(\Vert \Phi \Vert_{L^\infty(M)}) ^{-1} \Vert C_\Phi - C_{\Phi_0} \Vert_{L^2(\partial_+M)}\lesssim h(p_{\Phi}^n, p_{\Phi_0}^n) \lesssim \Vert C_\Phi - C_{\Phi_0} \Vert_{L^2(\partial_+M)}
\end{equation}
for a non-decreasing function $\omega:[0,\infty)\rightarrow [0,\infty)$ coming from \eqref{est2}. As we use the lower bound only on the event $\F'(A) = \{\Vert \Phi \Vert_{C^k(M)}\le A\}$, this adjustment is unproblematic, as $\omega$ can be controlled.\\
Finally we note that the proof of \eqref{smallf} is completely independent of the forward-estimates and only uses the upper bound in \eqref{hellingercomparison}. In particular \cite[Thm.\,5.13]{MNP19} can again be used to conclude \eqref{posconcrep}, as desired.
\end{proof}

\section{Appendix}

\subsection{Extensions}

Let $M$ be a compact manifold with boundary. By an `extension' of $M$ we mean a a larger manifold $N$ (of the same dimension) with interior containing $M$ as embedded sub-manifold. For example by gluing two copies of $M$ along the common boundary, one can always extend $M$ to a closed manifold.\\
If $N$ is an extension of $M$, then smooth functions and tensors on $M$ can themselves be extended to $N$ and one can ask them obey certain geometric or functional analytic properties. We record here two useful constructions:

\begin{lemma}[No return extension] \label{noreturn} Suppose $(M,g)$ is a compact Riemannian manifold with strictly convex boundary. Then there exists a complete extension $(N,g)$ with the property that geodesics that leave $M$ never re-enter and do not get trapped in $N\backslash M$. Precisely, if  $(x,v) \in \partial_-SM$ and $K\subset N$ is compact, then $\gamma_{x,v}(t) \in N\backslash M$ for all $t>0$ and $\gamma_{x,v}(t)\in N\backslash K$ for  $t\gg 1$ sufficiently large.
\end{lemma}

\begin{proof}  As a smooth manifold, $N$ is obtained by gluing $M$ and the  cylinder $[0,\infty)\times \partial M$ along $\partial M$. The metric on $M$ can then be extended smoothly to all of $N$ such that on $[0,\infty)_s\times \partial M$ it takes the form $\tilde g = \d s^2+ \tilde h_s$, where $(\tilde h_s)$ is a family of  Riemannian metrics on $\partial M$, depending smoothly on $s\ge 0$.\\ We now construct $(h_s)$, agreeing with $\tilde h_s$ for $s$ near zero, such that $g=\d s^2 + h_s$ satisfies the desired properties. First note that $2\partial_s \tilde h_s \vert_{s=0}$ is positive definite, as it coincides with the second fundamental form of $\partial M$. Thus by continuity there is an $\epsilon >0$ such that $\partial_s \tilde h_s$ is positive definite  for all $0\le s<2\epsilon $. Let $\xi_i:[0,\infty)\
\rightarrow [0,1]$ $(i=1,2)$ be smooth and monotonic  with $\xi_1+\xi_2=1$ and $1_{[0,\epsilon)} \le \xi_1 \le 1_{[0,2\epsilon)}$ and set 
$
h_s = \xi_1(s) \tilde h_s + s\xi_2(s)k,
$
where $k$ is a Riemannian metric on $\partial M$ that will be chosen later. We want to arrange that
\begin{equation} \label{noreturn1}
h_s > 0 \quad \text{ and }\quad  \mathbb{I}_s \equiv 2\partial_s h_s >0 \quad \text{ for all } s\ge 0,
\end{equation}
where $\mathbb{I}_s$ is the second fundamental form of $\{s\}\times \partial M\subset (N,g)$ and `$>$' is to be understood in the sense of positive-definiteness of symmetric bilinear forms on $T\partial M$. First note that, since $\xi_2(s)=0$ for $s<\epsilon$, we have $h_s\ge \xi_1(s) \tilde h_s +  \epsilon\xi_2(s)k > 0$ for all $s\ge 0$. Next, 
\begin{equation} \label{noreturn2}
\partial_s h_s =  \xi_2'(s)\left(s k    - \tilde h_s   \right)  + \left [\xi_1(s) \partial_s \tilde h_s  + \xi_2(s) k \right],
\end{equation}
and we can argue as follows: As $(\tilde h_s/s:\epsilon\le s \le 2\epsilon)$ is a \textit{compact} family of Riemannian metrics, it can be majorised by some $k$ in the sense that $sk-\tilde h_s \ge 0$ on $[\epsilon,2\epsilon]$. Hence, since  $\xi_2'$ is non-negative with support contained in $[\epsilon, 2\epsilon]$, the first term in \eqref{noreturn2} is non-negative. The second term is easily seen to be positive and thus \eqref{noreturn1} follows.\\
Let us verify that $(N,g)$  is indeed complete and has the no-return/non-trapping property. Take $p:N\rightarrow \R$  a smooth function, non-positive on $M$ and agreeing with projection onto the first factor on $[0,\infty)\times\partial M\subset N$. Then $p$ is proper and $\vert \d p \vert_g$ is bounded, which implies that $(N,g)$ must be complete. Further, the Hessian of $p$ on $[0,\infty)\times \partial M$ is given by the second fundamental form in \eqref{noreturn1} and thus $p$ is strictly convex. Then for $(x,v)\in \partial_-SM$ the function $q(t) = p(\gamma_{x,v}(t))$ $(t \ge 0)$ satisfies $q(0)=0$,  $q'(0)>0$ and further, as long as $q(t)\ge 0$, we must have $q''(t) = \mathbb{I}_{q(s)}[\dot \gamma(t),\dot \gamma(t)]\ge c>0$. This shows that $q(t)\ge 0$ for all $t\ge0$ and that $q$ is unbounded.  This immediately implies the no-return property ($\gamma(t)\in \{p\ge 0\}$ for $t\ge0$) and shows that $\gamma$ is not trapped.
\end{proof}

\begin{lemma}[Seeley, 1963] \label{seeley} Suppose $M$ is a compact manifold with boundary and $N$ is an extension. Then there exists a linear operator $E:C^\infty(M)\rightarrow C^\infty(N)$ which is continuous and has closed range in the all of the following functional settings:
\begin{equation}
E:H^s(M) \rightarrow H^s(N) ~(s\in \R),\quad E:C^k(M)\rightarrow C^k(N)~ (k\in \Z_{\ge  0}\cup \{\infty\})
\end{equation}
\end{lemma}

\subsection{Sobolev spaces} 
Let $M$ be a compact manifold (with or without boundary) of dimension $d\ge 1$ and $O\subset M$ an open set.  We collect here some well-known  results (interpolation inequality, metric entropy bound) concerning the Sobolev-spaces $H^s(O)$ ($s \in \R$), briefly discussing their proofs in the manifold case, which is avoided in many available references.\\
To avoid any notational ambiguity we first discuss our definition of $H^s(O)$, assuming the notion of $H^s(N)$ for a closed manifold $N$ to be known (cf.  \cite[Ch.\,4.3]{Tay11}). For $M$ a compact manifold {\it with} boundary we then let $H^s(M)=\{u=U\vert_{M^\interior}: U\in H^s(N) \}$, where $N$ is any closed extension of $M$. Similarly, elements in $H^s(O)$ are defined as restrictions (to $O^\interior$) of functions in $H^s(M)$.

\begin{lemma} \label{sequencespace} Suppose that $\partial M =\emptyset$. Then there are smooth functions $\varphi_k:M\rightarrow \R$, ($k=1,2,\dots$) such that for all $s\in \R$ an equivalent norm on the Sobolev-space $H^s(M)$ is given by
\begin{equation}
\Vert u \Vert_s^2 = \sum_{k\ge 1} k^{2s/d} \vert \langle u , \varphi_k \rangle \vert^2,\quad u\in H^s(M).
\end{equation}
\end{lemma}

\begin{proof} Let $g$ be a Riemannian metric on $M$, then the differential operator $1+\Delta_g$  has positive principal symbol and its spectrum consists of eigenvalues $0<\lambda_1^2\le \lambda_2^2 \le \dots \rightarrow \infty$. Let $\varphi_k$ ($k=1,2,\dots$) be the corresponding eigenfunctions (normalised to $\Vert \varphi_k \Vert_{L^2(M)} = 1$), then the Lemma follows from standard spectral theory. \\
Let us nevertheless sketch the main ideas leading to the result: For $u\in \D'(M)$ one writes $\hat u_k = \langle u ,\varphi_k\rangle$ for its Fourier-coefficients, and formally defines
\begin{equation}\label{defp}
P^s u \overset{\mathrm{def}}{=} \sum_{k\ge 1} \lambda_k^s \hat u_k \varphi_k, \quad s\in \R.
\end{equation}
A priori it is not clear that the operator $P^s$ is well defined, but the theory of complex powers of elliptic operators (cf. Theorem 10.1, Theorem 10.2, Proposition 10.3 and Theorem 11.2 in \cite{Shu78}) yields that $P^s$ is a classical, elliptic $\psi$do of order $s$ with
\begin{equation}
P^{2j}=(1+\Delta_g)^j \quad \text{ for } j\in \Z_{\ge 0},  \qquad P^sP^t=P^{s+t} \text{ for } s,t\in \R
\end{equation}
and the series in \eqref{defp} converges in $\D'(M)$. In particular $\Vert u \Vert'_s = \Vert P^s u \Vert_{L^2} =\left( \sum_{k\ge 1} \lambda_k^{2s} \vert \hat u_k \vert^2\right )^{1/2}$ defines a compatible norm on $H^s(M)$ and the result follows from the the asymptotic equivalence $\lambda_ k^2 \sim k^{2/d}$ (Proposition 13.1 in \cite{Shu78}).
\end{proof}

\begin{lemma}[Interpolation inequality] \label{interpol}
	Suppose $s_0<s_1$ and let  $s_\theta = (1-\theta) s_0 + \theta s_1$ ($\theta \in [0,1]$). Then for all $u\in H^{s_\theta}(O)$ we have
	\begin{equation}
	\Vert u \Vert_{H^{s_\theta}(O)} \le C \Vert u \Vert_{H^{s_0}(O)}^{1-\theta} \Vert u \Vert_{H^{s_1}(O)}^\theta
\end{equation}	 
for a constant $C>0$ only depending on $O,s_0,s_1$.
\end{lemma}

\begin{proof} Extend $M$ to a closed manifold $N$, such that $O\subset M\subset N$. Extend $u$ to a function $U\in H^{s_\theta}(N)$ and consider the following inequality (for the norms $\Vert \cdot \Vert_s$ on $H^s(N)$ defined in the previous Lemma):
\begin{equation*}
\begin{split}
\Vert U \Vert_{s_\theta}^2 &= \sum_{k\ge 1} \left( k^{2s_0/d} \vert \hat U_k \vert^2\right)^{1-\theta} \left( k^{2s_1/d} \vert \hat U_k \vert^2\right)^{\theta}\\
&\le \left(\sum_{k\ge 1} k^{2s_0/d} \vert \hat U_k \vert^2 \right)^{1-\theta}\left(\sum_{k\ge 1} k^{2s_1/d} \vert \hat U_k \vert^2 \right)^{\theta} = \Vert U\Vert_{s_0}^{2(1-\theta)} \Vert U \Vert_{s_1}^{2\theta}
\end{split}
\end{equation*}
Here $\hat U_k = \langle U,\varphi_k\rangle$ and we have used the H{\"o}lder-inequality for the exponents $1/(1-\theta)$ and $1/\theta$. This implies that $\Vert u \Vert_{H^{s_\theta}(O)} \le \Vert U\Vert_{s_0}^{(1-\theta)} \Vert U \Vert_{s_1}^{\theta}$ for all extensions $U$ and the Lemma follows by choosing $U = E u $ as in Lemma \ref{seeley}.
\end{proof}

Next, recall the notation $N(X,d,\epsilon)$ for the smallest number of $\epsilon$-balls needed to cover a (totally bounded) metric space $(X,d)$. Then:

\begin{lemma}[Metric entropy bound] \label{metricentropy}  Let $\mathbb{B}^s\subset H^s(M)$ ($s>0$) be the unit-ball. Then, as $\epsilon \rightarrow 0$, we have  $\log N(\mathbb{B}^s,\Vert \cdot \Vert_{L^2(M)},\epsilon) = O(\epsilon^{-s/d})$.
\end{lemma}

\begin{proof}
Using the representation as sequence space from Lemma \ref{sequencespace}, the Lemma is easily proved in the case $\partial M=\emptyset$ by the same arguments as in \cite[Theorem 4.3.36]{GiNi16}. The case $\partial M \neq \emptyset$ follows immediately by extending $M$ to a closed manifold $N$ and realising $H^s(M)$ as closed subspace of $H^s(N)$ via an extension operator $E$ as in Lemma \ref{seeley}.
\end{proof}

\begin{comment}

Similarly we define $C^k(M)$ ($k={0,\dots,\infty}$) via local charts if $\partial M=\emptyset$ and otherwise such that $C^k(N)\rightarrow C^k(M)$ is a surjection if $N$ is a closed extension of $M$. For all $k=0,\dots,\infty$ the space $C^k(M)$ has a natural Fr{\'e}chet-topology, which can be normed for $k<\infty$. As above,  $\Vert \cdot \Vert_{C^k(M)}$ will denote an unspecified choice of complete norm.\\
By a theorem of Seeley, if $N\supset M$, there exists a linear extension map  $E:C^\infty(M)\rightarrow C^\infty(N)$ which extends to a continuous map between $C^k$-spaces for all $k=0,\dots,\infty$.\\

If $O\subset M$ is an open subset, then we use the notation $H^s(O) $ and $C^k(O)$ to denote the following spaces:  In case of Sobolev-spaces we define $H^s(O)=\{u\vert_O: u\in H^s(M)\}$, topologised such that  the restriction map fits into a split exact sequence
\begin{equation}
0\rightarrow \{u\in H^s(M): \supp u \cap O =\emptyset\} \hookrightarrow H^s(M) \rightarrow H^s(O) \rightarrow 0
\end{equation}
of Hilbert-spaces.\\ %In particular, extension of an element $u\in H^s(O)$ across $O$ is always possible.\\
We define $C^k(O)$ to consist of maps $u:O\rightarrow \C$ with $\chi u\in C^k(M)$ for all cut-off's $\chi\in C^\infty(M)$ with $\supp \chi \subset O$ and equip it with the coarsest topology that makes the maps $u\mapsto \chi u\in C^k(M)$ continuous. In contrast to above, if $O\neq M$ then $C^k(O)$ is \textit{not} normable and $C^k(M)\rightarrow C^k(O)$ is {\it not} surjective.\\

\end{comment}

\bibliographystyle{plain}
\bibliography{ref.bbl}

\end{document}